\documentclass[11pt]{amsart}
\usepackage{geometry}                % See geometry.pdf to learn the layout options. There are lots.
\usepackage{graphicx}
\usepackage{amssymb,amsmath,amsthm,tikz,arydshln}
\usepackage{epstopdf}
\usepackage{xcolor}
\newmuskip\pFqmuskip

\newcommand*\pFq[6][8]{%
  \begingroup % only local assignments
  \pFqmuskip=#1mu\relax
  \mathcode`\,=\string"8000
  \begingroup\lccode`\~=`\,
  \lowercase{\endgroup\let~}\pFqcomma
  {}_{#2}F_{#3}{\left[\genfrac..{0pt}{}{#4}{#5};#6\right]}%
  \endgroup
}
\newcommand{\pFqcomma}{\mskip\pFqmuskip}

\newcommand*\pPq[6][8]{%
  \begingroup % only local assignments
  \pPqmuskip=#1mu\relax
  \mathcode`\,=\string"8000
  \begingroup\lccode`\~=`\,
  \lowercase{\endgroup\let~}\pPqcomma
  {}_{#2}\Phi_{#3}{\left[\genfrac..{0pt}{}{#4}{#5};#6\right]}%
  \endgroup
}
\newcommand{\pPqcomma}{\mskip\pPqmuskip}
\newtheorem{theorem}{Theorem}[section]
\newtheorem{lemma}[theorem]{Lemma}
\newtheorem{prop}[theorem]{Proposition}
\newtheorem{cor}[theorem]{Corollary}
\theoremstyle{definition}

\theoremstyle{remark}
\newtheorem{remark}[theorem]{Remark}

\usepackage{mathtools}

\numberwithin{equation}{section}

\newcommand{\be}{\begin{equation}}
\newcommand{\ee}{\end{equation}}
\newcommand{\ba}{\begin{eqnarray}}
\newcommand{\ea}{\end{eqnarray}}
\newcommand{\baa}{\begin{eqnarray*}}
\newcommand{\eaa}{\end{eqnarray*}}
\newcommand{\bea}{\begin{eqnarray*}}
\newcommand{\eea}{\end{eqnarray*}}
\newcommand{\bb}{\begin{thebibliography}}
\newcommand{\eb}{\end{thebibliography}}

      \def\cF{{\mathcal F}}
\def\cG{{\mathcal G}}      
      \def\cL{{\mathcal L}}
\def\cM{{\mathcal M}}

\title[Recurrence relations of X-LBP]
{Recurrence relations of Exceptional Laurent biorthogonal polynomials}
\author{Yu Luo}
\address{Department of Mathematics\\
Zhejiang University of Technology\\
Hangzhou 310014, China}
\email{luoyu4304@outlook.com}

\author{Satoshi Tsujimoto}
\address{Department of Applied Mathematics and Physics\\
Graduate School of Informatics, Kyoto University\\
Sakyo-Ku, Kyoto, 606 8501, Japan}
\email{tsujimoto.satoshi.5s@kyoto-u.jp}

\author{Hao Yang}
\address{Dongfeng Commercial Vehicle Technical Center of 
Dongfeng Commercial Vehicle Co., Ltd.\\
China}
\email{yang.hao.s83@kyoto-u.jp}

\thanks{Yu Luo\\
luoyu4304@outlook.com\\
+86-18720265091\\
Department of Mathematics, Zhejiang University of Technology, Hangzhou 310014, China\\
Satoshi Tsujimoto\\
tsujimoto.satoshi.5s@kyoto-u.jp\\
+81-75-753-5498\\
Department of Applied Mathematics and Physics, Graduate School of Informatics, Kyoto University, 
Sakyo-Ku, Kyoto, 606 8501, Japan\\
Hao Yang\\
yang.hao.s83@kyoto-u.jp\\
Department of Applied Mathematics and Physics, Graduate School of Informatics, Kyoto University, 
Sakyo-Ku, Kyoto, 606 8501, Japan\\
Dongfeng Commercial Vehicle Technical Center of Dongfeng Commercial Vehicle Co., Ltd., Wuhan, China\\
}

\begin{document}

\begin{abstract}
Exceptional extensions of a class of Laurent biorthogonal polynomials 
(the so-called Hendriksen-van Rossum polynomials) 
have been presented by the authors recently. 
This is achieved through Darboux transformations of generalized eigenvalue problems. 
%Specifically, exceptional Hendriksen-van Rossum (HR) polynomials are provided. 
In this paper, we discuss the recurrence relations satisfied by these exceptional Laurent biorthogonal polynomials 
and provide a type of recurrence relations with $3l_0+4$ terms explicitly, 
where the parameter $l_0$ corresponds to the degree of the polynomial part in the seed function used in the 
Darboux transformation. 
In the proof of these recurrence relations, the backward operator which maps an exceptional polynomial into a 
classical one plays a significant role. 

\smallskip
\noindent \textbf{\keywordsname}
Recurrence relations; 
Laurent biorthogonal polynomials; 
Exceptional Laurent biorthogonal polynomial; 
Hendriksen-van Rossum polynomials. 

\smallskip
\noindent \textbf{2020 Mathematics Subject Classification}
33C45; 33C47; 42C05
\end{abstract}

%\keywords{
%Laurent biorthogonal polynomials, Exceptional extensions, recurrence relations
%}
\maketitle

\section{Introduction}
An important topic that has attracted more and more attention from researchers 
in the field of mathematics and physics 
is called the exceptional-type extensions of classical orthogonal polynomials (COP), 
in which a sequence of polynomials with degree jumps can provide a basis for an appropriate weighted $L^2$-space. 
The exceptional orthogonal polynomials (XOP) generalize COP by loosening restrictions on their degree sequence. 

Over the past decade, significant efforts have been made in the theory and applications of XOP. 
The initial examples were introduced by G\'omez, Kamran, and Milson in \cite{GKM09, GKM10_1}, 
and were quickly recognized for their application potential by physicists and mathematicians
%After that, their application potential was immediately recognized by many physicists and mathematicians 
\cite{Quesne08, Quesne09, MR09, OS09, OS11}. 
These XOPs, satisfying second-order differential equations, can be used to derive new solvable potentials. 
The Darboux transformation is crucial in their construction, with many examples obtained using this method 
\cite{GKM10_2, GGM13, X q-Racah, STA}. 
It was later revealed that multiple-step or higher-order Darboux transformations lead to XOP labeled by multi-indices
\cite{GKM12, OS13}. 
Dur\'an also developed a systematic way of constructing XOP using the concept of dual families of polynomials 
\cite{BispOP, Krall_Hahn, XOPviaKrall}. 
Exceptional extensions of COP from the Askey scheme and Askey-Wilson scheme have been extensively studied, 
including the $q\rightarrow -1$ cases \cite{XBI}. 
Important properties such as recurrence relations \cite{GKK16, X-RR, O16}, 
zeros \cite{GMM13, HS12, KM15}
and spectral analysis \cite{LLM16} 
were also discussed. 
In \cite{X-Bochner}, a complete classification of the continuous XOP, 
generalizing classical ones (Jacobi, Laguerre, and Hermite polynomials), was addressed. 
Additionally, \cite{X-Kra} provides a detailed discussion of exceptional Krawtchouk polynomials. 
%%%%%

Recently, the authors of this paper \cite{XHR} introduced exceptional extensions of Hendriksen-van Rossum (HR) 
polynomials, known as exceptional Laurent biorthogonal polynomials (XLBP). 
These XLBPs are constructed through spectral transformations of generalized eigenvalue problems (GEVP), 
leading to Darboux transformations of GEVP. 
This results in four types of exceptional HR polynomials, including state-deletion, state-addition, and iso-spectral cases.

HR polynomials 
(also called Askey biorthogonal polynomials on the unit circle by \cite{ABP}) 
are defined by the following hypergeometric functions \cite{HR}: 
\begin{eqnarray}
\label{HR_P}
&&P_n(z):=P_n(z;\alpha,\beta)=\dfrac{(\beta)_n}{(\alpha+1)_n}
~_2F_1\left(
\begin{gathered}
-n, \alpha+1 \\
1-\beta-n
\end{gathered}
;z
\right),  \\
\label{HR_Q}
&&Q_n(z):=Q_n(z;\alpha,\beta)=P_n(z,\beta,\alpha), 
\end{eqnarray}
where $\alpha$ and $\beta$ are real parameters, 
$\{Q_n(z)\}_n$ are the biorthogonal partners of $\{P_n(z)\}_n$. 
Unless otherwise specified, throughout this paper, the symbol $P_n(z)$ ($Q_n(z)$) stands for 
$P_n(z;\alpha,\beta)$ ($Q_n(z;\alpha,\beta)$). 

The exceptional HR polynomials obtained from a single-step Darboux transformation \cite{XHR} are expressed as 
\begin{align}
\label{hphi_to_X1P}
P^{(j_0,l_0,n)}(z)=
\begin{cases}
\hat{\psi}^{(j_0,l_0,n)}(z), & j_0=1,2, \\
z^{l_0}\hat{\psi}^{(j_0,l_0,n)}(z), & j_0=3,4, 
\end{cases}
\end{align}
with 
\begin{align}
\label{hpsi_0}
\hat{\psi}^{(j_0,l_0,n)}(z)
%=P^{(j_0,l_0,n)}(z)
=Q^{(j_0)}(z)
\left(p^{(j_0)}_{l_0}(z)P'_{n}(z)-(p^{(j_0)}_{l_0}(z))'P_{n}(z)\right)-P^{(j_0)}(z)p^{(j_0)}_{l_0}(z)P_{n}(z), 
\end{align}
where the superscripts $j_0$ and $l_0$ stand for the type and the degree of the polynomial part of the seed function  (i.e., $p^{(j_0)}_{l_0}(z)$ for $j_0=1,2$ and $z^{l_0}p^{(j_0)}_{l_0}(z)$ for $j_0=3,4$), respectively,  
\begin{align}
\label{PQ_1}
&P^{(1)}(z)=0, \quad P^{(2)}(z)=\alpha+\beta, \quad P^{(3)}(z)=1+\alpha, \quad P^{(4)}(z)=-1+\beta+(1+\alpha)z, \\
\label{PQ_2}
&Q^{(1)}(z)=1, \quad Q^{(2)}(z)=1-z, \quad Q^{(3)}(z)=-z, \quad Q^{(4)}(z)=z(1-z),
\end{align}
and 
\begin{align}
\label{quasi_P_1}
&p^{(1)}_n(z)=P_n(z; \alpha,\beta),&
%&\xi_1(z)=1,&  &\theta^{(1)}_n=n; & \\
&p^{(2)}_n(z)=P_n(z; -\beta,-\alpha),& 
\\
\label{quasi_P_2}
%&\xi_2(z)=(1-z)^{-\alpha-\beta},& &\theta^{(2)}_n=n-\alpha-\beta; & \\
&p^{(3)}_n(z)=P_n(z^{-1}; \alpha,\beta),& \quad 
%&\xi_3(z)=(-z)^{-1-\alpha},& \quad &\theta^{(3)}_n=-n-1-\alpha-\beta; &\\
&p^{(4)}_n(z)=P_n(z^{-1}; -\beta,-\alpha).& \quad 
%&\xi_4(z)=(-z)^{-1+\beta}(1-z)^{-\alpha-\beta},& \quad 
%&\theta^{(4)}_n=-n-1. &
\end{align}
The degrees of $P^{(j_0,l_0,n)}(z)$ are given by 
\begin{align}
\label{deg_XP0}
\deg{P^{(j_0,l_0,n)}(z)}=n+l_0-\delta_{1,j}+\delta_{4,j}, 
\end{align}
where $\delta_{i,j}=1$ if $i=j$, and $\delta_{i,j}=0$ otherwise.

In this paper, we will discuss the recurrence relations satisfied by these four types of 
exceptional HR polynomials. 
Different from the cases of (classical) LBP, and similar to the cases of XOP, 
it turns out that XLBP satisfies longer recurrence relations. 
The type of recurrence relations we found is in the following shape, for large enough $n$, 
\begin{align*}
%\label{RR-XHR_0}
q(z)\sum^{l_0+1}_{l=0}a_l^{(j_0,l_0,n)}P^{(j_0,l_0,n-l)}(z) 
%=\sum^{n+l_0+1}_{j=\max\{n-l_0,l_0+1\}}b_j^{(j_0,l_0,n)}P^{(j_0,l_0,j)}(z), 
=\sum^{n+l_0+1}_{j=n-l_0}b_j^{(j_0,l_0,n)}P^{(j_0,l_0,j)}(z), 
\end{align*}
where $q(z)$ is a polynomial in $z$. 
More precisely, we have 
\begin{itemize}
\item[(0)] If $l_0=0$, then an exceptional HR polynomial returns to an HR polynomial, the recurrence relation is 
(already known, see (\ref{TTRR_HR})): 
\[
z(P_n(z)+b_nP_{n-1}(z))=P_{n+1}(z)+d_nP_n(z);
\]
\item[(1)] If $l_0\geq 1$ and  $n\geq k$, then the exceptional HR polynomials satisfy 
\[
q(z)\sum^{k}_{l=0}a_l^{(j_0,l_0,n)}P^{(j_0,l_0,n-l)}(z) 
=\sum^{n+l_0+1}_{j=0}b_j^{(j_0,l_0,n)}P^{(j_0,l_0,j)}(z); 
\]
\item[(2)] If $l_0\geq 1$ and $n\geq 2l_0+1$, then the exceptional HR polynomials satisfy 
\[
q(z)\sum^{l_0+1}_{l=0}a_l^{(j_0,l_0,n)}P^{(j_0,l_0,n-l)}(z) 
=\sum^{n+l_0+1}_{j=n-l_0}b_j^{(j_0,l_0,n)}P^{(j_0,l_0,j)}(z), 
\]
with the coefficients $a_l^{(j_0,l_0,n)}, b_j^{(j_0,l_0,n)}$ uniquely determined. 
\end{itemize}
Here the coefficients $\{a_l^{(j_0,l_0,n)}\}_{l=0,\ldots, l_0+1}$ and 
$\{b_j^{(j_0,l_0,n)}\}_{j=\max\{n-l_0,l_0+1\}, \ldots, n+l_0+1}$ 
do not depend on $z$. 
%In particular, if $l_0=0$ and $q^{(j_0)}_{0}(z)=z$, 
%then the recurrence relation (\ref{RR-XHR}) returns to that of HR polynomials (see (\ref{TTRR_HR})). 

In this paper, we will prove the case (1) formally and 
the case (2) explicitly. 
Here ``explicitly" means that the factor $q(z)$ and the coefficients $a_l^{(j_0,l_0,n)}$ will be presented explicitly. 

%%%% case (1)
\begin{theorem}
\label{thm_main_1}
If $l_0\geq 1$ and $n\geq k$, then the exceptional HR polynomials $\{P^{(j_0,l_0,n)}(z)\}_n$, $j_0\in\{1,2,3,4\}$, 
%defined by (\ref{hpsi_0}) and (\ref{hphi_to_X1P}) 
satisfy the following recurrence relations:  
\begin{align}
\label{RR-XHR_0}
q^{(j_0)}_{l_0}(z)\sum^{k}_{l=0}a_l^{(j_0,l_0,n)}P^{(j_0,l_0,n-l)}(z) 
=\sum^{n+l_0+1}_{j=0}b_j^{(j_0,l_0,n)}P^{(j_0,l_0,j)}(z), 
\end{align}
where $q^{(j_0)}_{l_0}(z)$ is a polynomial of degree $l_0+1$ defined by (\ref{qpi}), 
$a_l^{(j_0,l_0,n)}$, $l=0,\ldots,k$, and $b_j^{(j_0,l_0,n)}$, $j=0,\ldots,n+l_0+1$, are constants. 
%and $a_l^{(j_0,l_0,n)}$, $l=0,\ldots,l_0+1$, are given by proposition \ref{prop_Coe_a}  when $n\geq 2l_0+1$.
\end{theorem}
Note that the equation (\ref{RR-XHR_0}) is nontrivial for the exceptional polynomials since there are gaps in their degree 
sequence. 
%The proof of Theorem \ref{thm_main_1} follows from 
%Lemma (\ref{lem_Mpsi}), Lemma (\ref{lem_Mq}) and Lemma (\ref{lem_qpi}). 

More precisely, the coefficients $\{a_l^{(j_0,l_0,n)}\}$ and $\{b_l^{(j_0,l_0,n)}\}$ can be uniquely determined. 
%%%% case (2)
\begin{theorem}
\label{thm_main_2}
If $l_0\geq 1$ and $n\geq 2l_0+1$, then the exceptional HR polynomials $\{P^{(j_0,l_0,n)}(z)\}_n$, $j_0\in\{1,2,3,4\}$, 
%defined by (\ref{hpsi_0}) and (\ref{hphi_to_X1P}) 
satisfy the following recurrence relations:  
\begin{align}
\label{RR-XHR}
q^{(j_0)}_{l_0}(z)\sum^{l_0+1}_{l=0}a_l^{(j_0,l_0,n)}P^{(j_0,l_0,n-l)}(z) 
=\sum^{n+l_0+1}_{j=n-l_0}b_j^{(j_0,l_0,n)}P^{(j_0,l_0,j)}(z), 
\end{align}
where $q^{(j_0)}_{l_0}(z)$ is a polynomial of degree $l_0+1$ defined by (\ref{qpi}), 
the coefficients $a_l^{(j_0,l_0,n)}$, $l=0,\ldots,l_0+1$, are given by proposition \ref{prop_Coe_a}. 
\end{theorem}
If we omit the repetition of $P^{(j_0,l_0,m)}$'s on both sides, it is easy to count the total terms of this 
recurrence relation, which is $3l_0+4$.

This paper unfolds as follows. 
In section 2, we briefly review how to construct XLBP by using Darboux transformations of GEVP. 
Several important operators and related properties are also introduced. 
In section 3, the definition and classical properties of HR polynomials are presented. 
We also provide some essential properties such as the expansions of HR polynomials with twisted parameters 
and the expansion of the product of a polynomial with a certain degree and the linear combination of a sequence of 
HR polynomials, both of which contribute much to the proof of the recurrence relations of the 
exceptional HR polynomials. 
In section 4, we introduce the four types of exceptional HR polynomials and some basic properties. 
Section 5 contains the key idea of the proof of our main result. 
The backward operator which maps an exceptional HR polynomial to an HR polynomial plays an important role 
in our proof of the recurrence relations satisfied by the exceptional HR polynomials. 
In section 6, we provide several concrete examples of these recurrence relations explicitly. 
Finally, we give a concluding remark in section 7.

\section{Darboux transformation of a Generalized eigenvalue problem}
In this section, we provide a concise overview of the construction of 
exceptional Laurent biorthogonal polynomials \cite{XHR}, 
which involves the Darboux transformation of a GEVP and its adjoint problem. 
Certain operators will be crucial in establishing the recurrence relation. 
%(\ref{RR-XHR}).

Consider the GEVP 
\begin{align}
\label{GEVP_0}
\cL_1\psi(z)=\lambda\cL_2\psi(z), 
\end{align}
where $(\lambda, \psi(z) )$ is an eigen-pair of this GEVP, and 
\begin{align}
\label{opL}
& \cL_j = A_j(z) \partial_z^2 + B_j(z) \partial_z +C_j(z) I,\quad j=1,2,
%\label{opL2}
%& \cL_2 = A_2(z) \partial_z^2 + B_2(z) \partial_z +C_2(z) I,
\end{align}
with functions $A_j(z), B_j(z), C_j(z)$, $j=1,2$, of $z$, 
One can decompose $\cL_1$ and $\cL_2$ into 
\begin{align}
 & \cL_j = \widetilde{\phi}_j(z)\left(\cG_j \cF + I \right) , \quad j=1,2, 
%\qquad
% \cL_2 = \widetilde{\phi}_2(z)\left(\cG_2 \cF + I\right),
\end{align}
where
\begin{align}
\label{opG}
&\cG_j = \dfrac{1}{\widetilde{\phi}_j(z)\phi(z)}
\left(A_j(z)\partial_z+B_j(z)\right)\phi(z)\epsilon(z) I, \quad j=1,2, \\
%\label{opG2}
% \cG_2  &%= \dfrac{1}{\widetilde{\phi}_2(z)}\cG_2^{(0)}
%= \dfrac{1}{\widetilde{\phi}_2(z)\phi(z)}
%\left(A_2(z)\partial_z+B_2(z)\right)\phi(z)\epsilon(z) I\\
\label{opF}
&\cF = \dfrac{\phi(z)}{\varepsilon(z)}\partial_z\dfrac{1}{\phi(z)} I, \\
%\end{align}
%and 
%\begin{align}
\label{tphi}
&\widetilde{\phi}_j(z)=\dfrac{\cL_j\phi(z)}{\phi(z)}, \quad j= 1,2.
%\widetilde{\phi}_2(z)=\dfrac{\cL_2\phi(z)}{\phi(z)}. 
%=\dfrac{A_2(z)\phi''(z)+B_2(z)\phi'(z)+C_2(z)\phi(z)}{\phi(z)}.
\end{align}
The function $\phi(z)$ can be chosen as an eigenfunction of (\ref{GEVP_0}) with respect to an eigenvalue $\kappa$ 
such that
$
 \left(\cL_1 -  \kappa \cL_2\right)\phi(z) =0,
$
then 
$
 \widetilde{\phi}_1(z) = \kappa \widetilde{\phi}_2(z).
$
%Thus the GEVP (\ref{GEVP_0}) can be rewritten into 
%\begin{align}
% & \kappa \left(\cG_1 \cF + I\right) \psi(z) = 
%\lambda  \left(\cG_2 \cF + I\right)\psi(z).
%\end{align}

At the same time, consider the adjoint GEVP for the inner product on the unit circle defined as:  
\begin{align}
\label{inner_0}
\langle f(z), g(z) \rangle
:=\dfrac{1}{2\pi}\int^{2\pi}_{0}f(e^{ix})\overline{g(e^{ix})}dx
=\dfrac{1}{2\pi}\int^{2\pi}_{0}f(e^{ix})g(e^{-ix})dx.
\end{align}
It was shown in \cite{XHR} (Lemma 3.4.) that the adjoint of the operator $C(z)I$ is 
\begin{align}
\label{ha_C0}
&
\overline{(C(z)I)^{\ast}}=C(z^{-1})I, 
\end{align}
while the formal adjoints of the operators $B(z)\partial_z$ and $A(z)\partial_z^2$ are 
\begin{align}
\label{ha_C}
&
%\overline{(C(z)I)^{\ast}}=C(z^{-1})I, \hspace{2mm}
\overline{(B(z)\partial_z)^{\ast}}
=z\partial_z B(z^{-1})z I, \quad
\overline{(A(z)\partial_z^2)^{\ast}}
=z^2\partial_z^2 A(z^{-1})z^2 I, 
\end{align}
where $A(z)\in C^{2}$, $B(z)\in C^{1}$ and $C(z)$ are arbitrary functions with real parameters. 

The adjoint GEVP of (\ref{GEVP_0}) is 
\begin{align}
\label{aGEVP}
\overline{\cL_1^{\ast}}\psi^{\ast}(z) = \tau \overline{\cL_2^{\ast}}\psi^{\ast}(z),
\end{align}
with 
\begin{align}
\label{opL_ast}
\overline{\cL_j^{\ast}} &= z^2\partial_z^2 A_{j}(z^{-1})z^2+z\partial_z B_{j}(z^{-1})z+C_{j}(z^{-1})I,  \quad j=1,2, 
\end{align}
where $(\tau, \psi^{\ast}(z) )$ is an eigen-pair of the adjoint GEVP. 
Similarly, the adjoint operators can be decomposed into 
\begin{align}
\label{decom_haL}
\overline{\cL_j^{\ast}}&= \left(\overline{\cF^{\ast}} \overline{\cG_j^{\ast}} + I \right) \overline{\widetilde{\phi}_i(z)} I, 
\quad j=1,2, 
\end{align}
where 
\begin{align}
\label{ophaF}
&\overline{\mathcal{F}^{\ast}}
=z\partial_z\dfrac{z}{\epsilon(z^{-1})} I-\dfrac{\phi'(z^{-1})}{\epsilon(z^{-1})\phi(z^{-1})} I, \\
\label{ophaG}
&\overline{\mathcal{G}^{\ast}_i}
=z\partial_z\dfrac{zA_i(z^{-1})\epsilon(z^{-1})}{\widetilde{\phi}_i(z^{-1})} I
+\epsilon(z^{-1})\dfrac{B_i(z^{-1})+A_i(z^{-1})\dfrac{\epsilon'(z^{-1})}{\epsilon(z^{-1})}
+A_i(z^{-1})\dfrac{\phi'(z^{-1})}{\phi(z^{-1})}}{\widetilde{\phi}_i(z^{-1})} I. 
\end{align}
%The above follows from the fact that (when $z$ is real) the adjoint operator of $\partial_z$ is $-\partial_z$. 

By introducing the transformed functions
\begin{align}
\label{DT_f1}
 \widehat{\psi}(z) = {\cF}{\psi}(z), \qquad
 \widehat{\psi^{\ast}}(z) = (\kappa \overline{\cG_1^*} - \tau\overline{\cG_2^*})\overline{\widetilde{\phi}_2(z)}\psi^{\ast}(z), 
\end{align}
the GEVP (\ref{GEVP_0}) and the adjoint GEVP (\ref{aGEVP}) can be rewritten into 
\begin{align}
\label{backward_hpsi}
 & (\kappa \cG_1-\lambda \cG_2) \widehat{{\psi}}(z) = (\lambda - \kappa) {\psi(z)}, \\
\label{backward_hchi}
 &\overline{\cF^{\ast}} \widehat{\psi^{\ast}}(z) = (\tau - \kappa) \overline{\widetilde{\phi}_2(z)}{\psi^{\ast}(z)}, 
\end{align}
where the left-hand sides can be seen as backward operators from $\widehat \psi(z)$ to $\psi(z)$ 
and from $\widehat \psi^{\ast}(z)$ to $\widetilde\phi_2(z) \psi^{\ast}(z)$, respectively.
Moreover, by applying $\cF$ and $\kappa \overline{\cG_1^{\ast}}- \tau \overline{\cG_2^{\ast}}$  to 
(\ref{backward_hpsi}) and (\ref{backward_hchi}), respectively, one obtains
\begin{align}
\label{mGEVP_00}
 & {\cF} (\kappa \cG_1-\lambda \cG_2) \widehat{{\psi}}(z) = (\lambda - \kappa) \widehat{{\psi}}(z), \\
\label{maGEVP_00}
 &(\kappa \overline{\cG_1^{\ast}}- \tau \overline{\cG_2^{\ast}}) \overline{\cF^{\ast}} \widehat{\psi^{\ast}}(z) = (\tau - \kappa) \widehat{\psi^{\ast}}(z), 
\end{align}
which lead to 
the Darboux-transformed GEVPs corresponding to (\ref{GEVP_0}) and (\ref{aGEVP}): 
\begin{align}
\label{XL1_ast}
 \left(\widehat{\cL}_1 - \lambda \widehat{\cL}_2\right) \widehat{\psi}(z) = 0, \qquad
 \left(\widehat{\cL}_1^* - \tau \widehat{\cL}_2^*\right) \widehat{\psi^{\ast}}(z) = 0.
\end{align}
where 
\begin{align}
\label{factor_L0}
&\widehat{\cL}_1 = \kappa\left({\cF} \cG_1 + I\right),
\hspace{6mm} \widehat{\cL}_2 = {\cF} \cG_2 + I,  \\
\label{factor_L1}
&\widehat{\cL^{\ast}_1}=\kappa(\overline{\cG^{\ast}_1\cF^{\ast}}+I),  \quad
\widehat{\cL^{\ast}_2}=\overline{\cG^{\ast}_2\cF^{\ast}}+I. 
\end{align}

The biorthogonality relation of $\psi$ and $\overline{\cL_2^{\ast}}\psi^{\ast}$:
\begin{align}
\label{biorth_0}
 \langle \psi, \overline{\cL_2^{\ast}}\psi^{\ast} \rangle = h \delta_{\lambda,\tau},
\end{align}
can be formally obtained from the GEVPs \eqref{GEVP_0} and \eqref{aGEVP} as follows
\begin{align}
\label{adness}
(\tau-\lambda)\langle \psi, \overline{\cL_2^{\ast}}\psi^{\ast} \rangle
=\langle \psi, (\overline{\cL_1^{\ast}}-\lambda\overline{\cL_2^{\ast}})\psi^{\ast} \rangle
=\langle (\cL_1-\lambda\cL_2)\psi, \psi^{\ast} \rangle=0. 
\end{align}
where $h$ is a constant. 
%which we call {\em biorthogonality} of $\psi$ and $\cL_2^* \chi$ and is shown by

Similarly, the biorthogonality relation of $\widehat{\psi}$ and $\widehat{\cL_2^{\ast}}\widehat{\psi^{\ast}}$:
\begin{align}
\label{Xbiorth}
\langle \widehat{\psi}, \widehat{\cL}_2^*\widehat{\psi^{\ast}}\rangle 
=(\tau - \kappa) \langle \psi, \cL_2^*\psi^{\ast} \rangle
=(\tau - \kappa)  h \delta_{\lambda,\tau} \quad (\mbox{if }\lambda \ne \kappa),
\end{align}
follows from 
\begin{align}
\label{adness_X}
 \langle \widehat{\psi}, \widehat{\cL}_2^*\widehat{\psi^{\ast}} \rangle
 = \langle \cF \psi, \kappa(\overline{\cG_1^{\ast}} -\overline{\cG_2^{\ast}})\overline{\widetilde{\phi}_2}\psi^{\ast} \rangle
 = \langle \psi, (\overline{\cL_1^{\ast}} -\kappa \overline{\cL_2^{\ast}})\psi^{\ast} \rangle
 = (\tau - \kappa )\langle \psi, \overline{\cL_2^{\ast}}\psi^{\ast} \rangle. 
\end{align}
%where (\ref{decom_haL}), (\ref{factor_L1}) and the following relations have been used
%\begin{align*}
%  \widehat{\cL}_2^{\ast} \widehat{\psi^{\ast}}
% &= \left(\overline{\cG^{\ast}_2\cF^{\ast}} + I\right)
% (\kappa \overline{\cG_1^{\ast}} - \tau\overline{\cG_2^{\ast}})\overline{\widetilde{\phi}_2}\psi^{\ast}
% =\overline{\cG^{\ast}_2}(\kappa\overline{\cF^{\ast}\cG^{\ast}_1}-\tau\overline{\cF^{\ast}\cG^{\ast}_2})
% \overline{\widetilde{\phi}_2}\psi^{\ast}
%+(\kappa \overline{\cG_1^{\ast}} - \tau\overline{\cG_2^{\ast}})\overline{\widetilde{\phi}_2}\psi^{\ast} \\
% &= (\tau-\kappa)\overline{\cG_2^{\ast}}\overline{\widetilde{\phi}_2}\psi^{\ast}
%+(\kappa \overline{\cG_1^{\ast}} - \tau\overline{\cG_2^{\ast}})\overline{\widetilde{\phi}_2}\psi^{\ast}
%= \kappa(\overline{\cG_1^{\ast}} -\overline{\cG_2^{\ast}})\overline{\widetilde{\phi}_2}\psi^{\ast}. 
%\end{align*}

\subsection{quasi-Laurent-polynomial eigenfunctions}
A general class of eigenfunctions of a GEVP is called quasi-Laurent-polynomial eigenfunctions:
\[
 \left(\cL_1 -  \kappa \cL_2\right)\phi(z) =0, \quad \phi(z)=\xi(z)p(z), 
\] 
where $\xi(z)$ is a gauge function and $p(z)$ is a Laurent polynomial. 

\begin{lemma}[\cite{XHR}]
\label{lem_quasi}
There are four classes of quasi-Laurent-polynomial eigenfunctions of the GEVP defined by (\ref{opL}): 
\begin{align}
\label{quasi}
(L_1-\theta^{(j)}_nL_2)\phi^{(j,n)}(z)=0, \quad \phi^{(j,n)}(z)=\xi_j(z)p^{(j)}_n(z), \quad
j=1,2,3,4, 
\end{align}
which can be expressed in terms of: 
\begin{align*}
&p^{(1)}_n(z)=P_n(z; \alpha,\beta),&
&\xi_1(z)=1,&  &\theta^{(1)}_n=n; & \\
&p^{(2)}_n(z)=P_n(z; -\beta,-\alpha),& 
&\xi_2(z)=(1-z)^{-\alpha-\beta},& &\theta^{(2)}_n=n-\alpha-\beta; & \\
&p^{(3)}_n(z)=P_n(z^{-1}; \alpha,\beta),& \quad 
&\xi_3(z)=(-z)^{-1-\alpha},& \quad &\theta^{(3)}_n=-n-1-\alpha-\beta; &\\
&p^{(4)}_n(z)=P_n(z^{-1}; -\beta,-\alpha),& \quad 
&\xi_4(z)=(-z)^{-1+\beta}(1-z)^{-\alpha-\beta},& \quad 
&\theta^{(4)}_n=-n-1. &
\end{align*}
\end{lemma}

The adjoint problem of (\ref{quasi}) is
\begin{equation}
\label{haGEVP_HR}
(\overline{L_1^{\ast}}-\mu_n\overline{L^{\ast}_2})\phi^{\ast}(z)=0, 
\end{equation}
where 
\begin{align}
\label{ophaL1}
 &\overline{L^{\ast}_1}
 =z^2(z-1)\partial_z^2-z(2+\alpha+(-3+\beta)z)\partial_z+(1-\beta)z I, \\
\label{ophaL2}
 &\overline{L^{\ast}_2}
 =z(z-1)\partial_z+(-1-\alpha+z) I. 
\end{align}

\begin{lemma}[\cite{XHR}]
\label{lem_haquasi}
There are four classes of quasi-Laurent-polynomial eigenfunctions of the adjoint GEVP (\ref{haGEVP_HR})
\[
(\overline{L_1^{\ast}}-\mu_n\overline{L^{\ast}_2})\phi^{(j,n)\ast}(z)=0, \quad 
\phi^{(j,n)\ast}(z)
=w_j(z)\tilde{p}^{(j)}_n(z), \quad j=1,2, 3,4,
\]
where
\begin{align*}
&\tilde{p}^{(1)}_n(z)=P_n(z; \beta-1, \alpha+1),& \hspace{1mm} 
&w_1(z)=(1-z)^{\alpha+\beta}(-z)^{-1-\alpha},& \hspace{1mm} 
&\mu^{(1)}_n=n; & \\
&\tilde{p}^{(2)}_n(z)=P_n(z;-\alpha-1,-\beta+1),&  \hspace{1mm} 
&w_2(z)=(-z)^{-1-\alpha},& \hspace{1mm} 
&\mu^{(2)}_n=n-\alpha-\beta; & \\
&\tilde{p}^{(3)}_n(z)=P_n(z^{-1}; \beta-1, \alpha+1),& \hspace{1mm} 
&w_3(z)=(1-z)^{\alpha+\beta}(-z)^{-1-\alpha-\beta}, &\\
&\hspace{4.7mm}\mu^{(3)}_n=-n-1-\alpha-\beta; \nonumber \\
&\tilde{p}^{(4)}_n(z)=P_n(z^{-1};-\alpha-1,-\beta+1),& \hspace{1mm} 
&w_4(z)=z^{-1},& \hspace{1mm} 
&\mu^{(4)}_n=-n-1. &
\end{align*}
\end{lemma}

\section{Hendriksen-van Rossum polynomials}
%HR polynomials are defined by \cite{HR}: 
%\begin{eqnarray}
%\label{HR_P}
%&&P_n(z):=P_n(z;\alpha,\beta)=\dfrac{(\beta)_n}{(\alpha+1)_n}
%~_2F_1\left(
%\begin{gathered}
%-n, \alpha+1 \\
%1-\beta-n
%\end{gathered}
%;z
%\right),  \\
%\label{HR_Q}
%&&Q_n(z):=Q_n(z;\alpha,\beta)=P_n(z,\beta,\alpha), 
%\end{eqnarray}
%where $\alpha, \beta$ are real parameters, 
%and $\{Q_n(z)\}_n$ are the biorthogonal partners of $\{P_n(z)\}_n$. 
%Unless otherwise specified, throughout this paper, the symbol $P_n(z)$ ($Q_n(z)$) stands for 
%$P_n(z;\alpha,\beta)$ ($Q_n(z;\alpha,\beta)$). 
In this section, we present key properties of HR polynomials, 
including both previously known results from references \cite{XHR, HR, ABP}, 
and original findings by the authors. 
%The definition of the Gauss hypergeometric function is 
%\[
%~_rF_s\left(
%\begin{gathered}
%a_1, a_2, \ldots a_r \\
%b_1, b_2, \ldots, b_s
%\end{gathered}
%;z
%\right)=
%\sum^{\infty}_{n=0}
%\dfrac{(a_1,a_2,\ldots,a_r)_n}{n!(b_1,b_2,\ldots,b_s)_n}z^n, 
%\]
%where $(a)_n$ is the shifted factorial (or the standard Pochhammer symbol) defined by
%\[
%(a)_0=1, \quad
%(a)_n:=a(a+1)\cdots(a+n-1), \quad n=1,2,\ldots
%\]

HR polynomials satisfy the recurrence relation: 
\begin{equation}
\label{TTRR_HR}
P_{n+1}(z)+d_nP_n(z)=z(P_n(z)+b_nP_{n-1}(z)), 
\end{equation}
which can be seen as a GEVP with eigenvalue $z$, 
and  
\begin{equation}
\label{TTRR_Coe_HR}
d_n:=
d_n(\alpha,\beta)
%=-\dfrac{P_{n+1}(0)}{P_n(0)}
=-\dfrac{n+\beta}{n+\alpha+1}, \quad
b_n:=
b_n(\alpha,\beta)
%=P_{n+1}(0)Q_n(0)-\dfrac{P_{n+1}(0)}{P_n(0)}
=-\dfrac{n(n+\alpha+\beta)}{(n+\alpha)(n+\alpha+1)}. 
\end{equation}

HR polynomials are polynomial solutions of the linear second-order differential equation: 
\begin{equation}
\label{Deq_HR}
\left[z(1-z)\partial_z^2+(1-\beta-n-(2+\alpha-n)z)\partial_z\right]P_n(z)=-n(\alpha+1)P_n(z). 
\end{equation}
This equation can also be rewritten into a GEVP: 
%HR polynomials also satisfy the following GEVP 
\begin{equation}
\label{GEVP_HR}
L_1P_n(z)=\theta_nL_2P_n(z), \quad \theta_n=n, 
%\left[z\partial_z^2+(1-\beta-n)\partial_z+n(\alpha+1)\right]P_n(z)=z\left[z\partial_z^2+(2+\alpha-n)\partial_z\right]P_n(z). 
\end{equation}
with 
\begin{align}
\label{opL}
L_1=A_1(z)\partial_z^2+B_1(z)\partial_z, \quad L_2=B_2(z)\partial_z+C_2(z)I, 
\end{align}
where 
\begin{align}
\label{opL_ABC}
A_1(z)=z(1-z), \hspace{1.2mm} B_1(z)=(1-\beta-(2+\alpha)z), \hspace{1.2mm} 
B_2(z)=(1-z), \hspace{1.2mm} C_2(z)=-(\alpha+1). 
\end{align}
And it turns out that the operators $L_1$ and $L_2$ shift the parameters of $P_n(z)$ as follows: 
\begin{align}
\label{opL1_P}
&L_1P_n(z)=-n(n+\alpha+1)P_n(z; \alpha+1, \beta-1), \\
\label{opL2_P}
&L_2P_n(z)=-(n+\alpha+1)P_n(z; \alpha+1, \beta-1). 
\end{align}

The biorthogonality relation on the unit circle satisfied by HR polynomials is: 
\begin{equation}
\label{Biorth_HR}
%\dfrac{1}{2\pi i}\int_{|z|=1}P_n(z)Q_m(1/z)w(z)\dfrac{dz}{z}
\langle w(z)P_n(z), Q_m(z) \rangle
=\dfrac{1}{2\pi}\int_{0}^{2\pi}P_n(e^{ix})Q_m(e^{-ix})w(e^{ix})dx
=h_n\delta_{mn}, 
\end{equation}
where  
\begin{equation}
\label{w_HR}
w(z)=(-z)^{-\beta}(1-z)^{\alpha+\beta}, \quad
h_n=\dfrac{(\alpha+n+1)_{\infty}(\beta+n+1)_{\infty}}{(n+1)_{\infty}(\alpha+\beta+n+1)_{\infty}},
\end{equation}
and the branch of $(-z)^{-\beta}$ and of $(1-z)^{\alpha+\beta}$ is chosen as follow: 
\begin{eqnarray*}
&&(-z)^{-\beta}=|z|^{-\beta} \text{ if arg}z=\pi (0<\text{arg}z<2\pi), \\
&&(1-z)^{\alpha+\beta}=|1-z|^{\alpha+\beta}\text{ if arg}(1-z)=0 (-\pi<\text{arg}(1-z)<\pi). 
\end{eqnarray*}
Alternatively, the norming constant $h_n$ can also be expressed in terms of Gamma function: 
\begin{align}
\label{h_n_Gamma}
h_n=\dfrac{(1)_n\Gamma(n+\alpha+\beta+1)}{\Gamma(n+\alpha+1)\Gamma(n+\beta+1)}. 
\end{align}

%An analog of the Christoffel-Darboux formula for LBP is given (\cite{PASI}, Proposition 1.) by 
%\begin{equation}
%\label{CD_formula}
%\dfrac{P_{n+1}(x)Q_n(1/y)-(x/y)^nP_{n+1}(y)Q_n(1/x)}{\tilde{h}_n}
%=(x-y)\sum^{n}_{k=0}\dfrac{P_k(x)Q_k(1/y)}{\tilde{h}_k},
%\end{equation}
%where 
%\begin{equation}
%\label{CD_formula_h}
%\tilde{h}_0=1, \quad
%\tilde{h}_n=\prod^{n-1}_{k=0}(1-P_{k+1}(0)Q_{k+1}(0)),  \quad n=1,2,\ldots
%\end{equation}
%For the HR polynomials, we have 
%\[
%%P_{n}(0)=\dfrac{(\beta)_n}{(\alpha+1)_n}, \quad
%%Q_{n}(0)=\dfrac{(\alpha)_n}{(\beta+1)_n}, \quad
%\tilde{h}_n=\dfrac{(1)_n(\alpha+\beta+1)_n}{(\alpha+1)_n(\beta+1)_n}
%=\dfrac{(1)_{\infty}(\alpha+\beta+1)_{\infty}}{(\alpha+1)_{\infty}(\beta+1)_{\infty}}h_n, 
%\]
%which corresponds to the weight function $w(z)$ with a constant multiplier. 

The positive-definiteness of the linear functional defined by (\ref{Biorth_HR}) is equivalent to $h_n>0, n=0,1,\ldots$, 
which can be guaranteed by a stronger condition: 
\begin{equation}
\label{cond_conv_w}
\alpha>-1, \quad \beta>-1, \quad \alpha+\beta>-1. 
\end{equation}

The moments related to the weight function $w(z)$ are
\[
c_n=\dfrac{1}{2\pi i}\int_{|z|=1}z^nw(z)\dfrac{dz}{z}
=\dfrac{\sin{(\beta\pi)}\Gamma(n-b)}{\sin{((\alpha+\beta)\pi)}\Gamma(-(\alpha+\beta))\Gamma(1+n+\alpha)}, \quad
n=0,1,2,\ldots
\]
under the condition $\Re{(\alpha+\beta)}>-1$, which is equivalent to $\alpha+\beta>-1$ since $\alpha, \beta$ are real. 

%We also find the relationship between the coefficients of $L_1$ and the weight function $w(z)$: 
%\begin{equation}
%\label{w_AB}
%\dfrac{w'(z)}{w(z)}=-\dfrac{\alpha z+\beta}{z(1-z)}
%=\dfrac{B_1(z)-A'_1(z)}{A_1(z)} {\color{red}, }
%\end{equation}
%which can be rewritten into a Pearson equation
The coefficients of $L_1$ and the weight function $w(z)$ satisfy the following Pearson equation
\begin{equation}
\label{w_AB_r}
(A_1(z)w(z))'=B_1(z)w(z). 
\end{equation}
%This Pearson equation is indeed the same one as in the case of classical orthogonal polynomials. 

The authors have already provided several useful formulas in \cite{XHR}. 
Here, we introduce two additional formulas and consolidate them into the following lemma.
\begin{lemma}
\label{formula_HR}
The following relations are satisfied by the HR polynomials: 
\begin{align}
\label{formula_HR_1}
&z^nP_n(z^{-1})=\dfrac{(\beta)_n}{(\alpha+1)_n}P_n(z; \beta-1,\alpha+1), \\
\label{formula_HR_2}
&P'_n(z)=nP_{n-1}(z;\alpha+1,\beta), \\
%&zP'_n(z;\beta-1,\alpha+1)
%=n\left[P_n(z; \beta-1,\alpha+1)-\dfrac{(1+\alpha)}{(n-1+\beta)}P_{n-1}(z; \beta-1,\alpha+2)
%\right], \\
\label{formula_HR_3}
&\dfrac{P'_n(z;\beta-1,\alpha+1)}{P_n(z;\beta-1,\alpha+1)}
=\dfrac{n}{z}\left[1-\dfrac{(1+\alpha)}{(n-1+\beta)}\dfrac{P_{n-1}(z;\beta-1,\alpha+2)}{P_n(z;\beta-1,\alpha+1)}\right], \\
\label{formula_HR_4}
&\dfrac{P'_n(z; -\alpha-1,-\beta+1)}{P_n(z; -\alpha-1,-\beta+1)}
=\dfrac{n}{z}\left[1-\dfrac{P_{n-1}(z^{-1}; -\beta+1,-\alpha)}{zP_n(z^{-1}; -\beta,-\alpha)}\right], \\
&
\label{form_HR_0}
zP_n(z)=P_{n+1}(z)+(d_n-b_n)P_{n}(z; \alpha-1, \beta+1), \\
&
\label{form_HR_1}
(z-1)P'_n(z)=n\left(P_n(z)-\frac{n+\alpha+\beta}{n+\alpha}P_{n-1}(z)\right). 
\end{align}
\end{lemma}
By using (\ref{formula_HR_2}), the formula (\ref{form_HR_1}) can also be written as: 
\begin{align}
\label{form_HR_2}
(z-1)P_{n-1}(z;\alpha+1,\beta)=P_n(z)-\frac{n+\alpha+\beta}{n+\alpha}P_{n-1}(z). 
\end{align}
\begin{remark}
One can immediately find the ladder operators of HR polynomials by using the relations (\ref{formula_HR_2}) 
and (\ref{opL1_P}) as follows 
\begin{align*}
&
\text{Lowering operator: } \quad
\partial_z P_n(z)=nP_{n-1}(z; \alpha+1, \beta), \\
&
\text{Raising operator: } \quad
\left(A_1(z)\partial_z+B_1(z)
\right)P_n(z;\alpha+1, \beta)=-(n+\alpha+2)P_{n+1}(z;\alpha+1, \beta-1). 
\end{align*}

Moreover, it follows from (\ref{formula_HR_2}) that 
\begin{align}
\label{int_P}
%\int P_{n}(z) dz=(P_{n+1}(z;\alpha-1,\beta)-P_{n+1}(0;\alpha-1,\beta))/(n+1).
\int P_{n}(z) dz=P_{n+1}(z;\alpha-1,\beta)/(n+1)+Const.
\end{align}
\end{remark}

%Some formulas related to the HR polynomials: 
%\begin{align}
%&
%\label{form_HR_0}
%zP_n(z)=P_{n+1}(z)+(d_n-b_n)P_{n}(z; \alpha-1, \beta+1), \\
%&
%(1-z)P'_n(z)=n\left(P_n(z)-\frac{n+\alpha+\beta}{n+\alpha}P_{n-1}(z)\right)
%\end{align}

We also found the following type of recurrence relations of HR polynomials, where the coefficients are polynomials in $z$ 
instead of constants. 
%%%%% lemma
\begin{lemma}
\label{lem_expan_P_n}
The HR polynomials $P_n(z)$ satisfy the following recurrence relation: 
\begin{align}
\label{recurr_P}
P_{n+k+1}(z)=D_{k}(z)P_{n+1}(z)+B_{k}(z)P_{n}(z), 
\end{align}
where $D_{k}(z)$ and $B_{k}(z)$ are polynomials of degree $k$ which are defined by  
\begin{align}
\label{recurr_P_D}
&D_{k}(z)=(z-d_{n+k})D_{k-1}(z)+b_{n+k}zD_{k-2}(z), \\
\label{recurr_P_B}
&B_{k}(z)=(z-d_{n+k})B_{k-1}(z)+b_{n+k}zB_{k-2}(z), \quad k\geq 1,
\end{align}
where $D_{0}(z)=1$, $D_{i}(z)=0$ for $i<0$, 
and $B_{0}(z)=0$, $B_{-1}(z)=1$, $B_{j}(z)=0$ for $j<-1$. 
Moreover, one easily finds that $(b_{n+1}z)^{-1}B_{k}(z)$ is a monic polynomial, $k\geq 1$. 
\end{lemma}
%%%%% proof
\begin{proof}
It follows from the recurrence relation (\ref{TTRR_HR}) that 
\[
P_{n+2}(z)=(z-d_{n+1})P_{n+1}(z)+b_{n+1}zP_{n}(z), 
\]
which can be rewritten as the following matrices multiplication: 
\begin{align*}
\begin{pmatrix}
P_{n+2}(z) \\
P_{n+1}(z)
\end{pmatrix}
=
\begin{pmatrix}
z-d_{n+1} & b_{n+1}z \\
1 & 0
\end{pmatrix}
\begin{pmatrix}
P_{n+1}(z) \\
P_{n}(z)
\end{pmatrix}.
\end{align*}
Inductively, we can conclude that for $k\geq 1$, 
\begin{align*}
\begin{pmatrix}
P_{n+k+1}(z) \\
P_{n+k}(z)
\end{pmatrix}
=
\begin{pmatrix}
z-d_{n+k} & b_{n+k}z \\
1 & 0
\end{pmatrix}
\cdots
\begin{pmatrix}
z-d_{n+1} & b_{n+1}z \\
1 & 0
\end{pmatrix}
\begin{pmatrix}
P_{n+1}(z) \\
P_{n}(z)
\end{pmatrix}. 
\end{align*}
If we write 
\begin{align*}
\begin{pmatrix}
P_{n+k+1}(z) \\
P_{n+k}(z)
\end{pmatrix}
=
\begin{pmatrix}
D_k(z) &B_k(z) \\
D_{k-1}(z) & B_{k-1}(z)
\end{pmatrix}
\begin{pmatrix}
P_{n+1}(z) \\
P_{n}(z)
\end{pmatrix}, 
\end{align*}
then the following relation leads to (\ref{recurr_P_D}) and (\ref{recurr_P_B}):
\begin{align*}
\begin{pmatrix}
D_{k}(z) &B_{k}(z) \\
D_{k-1}(z) & B_{k-1}(z)
\end{pmatrix}
=
\begin{pmatrix}
z-d_{n+k} & b_{n+k}z \\
1 & 0
\end{pmatrix}
\begin{pmatrix}
D_{k-1}(z) &B_{k-1}(z) \\
D_{k-2}(z) & B_{k-2}(z)
\end{pmatrix}. 
\end{align*}

\end{proof}
Here we list the first few terms of $D_k(z)$ and $B_k(z)$ for reference: 
\begin{align*}
&
D_{1}(z)=z-d_{n+1}, \quad B_{1}(z)=b_{n+1}z;  \\
&
D_{2}(z)=(z-d_{n+2})(z-d_{n+1})+b_{n+2}z, \quad B_{2}(z)=b_{n+1}z(z-d_{n+2}); \\
&
D_{3}(z)=(z-d_{n+3})[(z-d_{n+2})(z-d_{n+1})+b_{n+2}z]+b_{n+3}z(z-d_{n+1}), \\
&
B_{3}(z)=b_{n+1}z[(z-d_{n+3})(z-d_{n+2})+b_{n+3}z]; \quad \cdots
\end{align*}

%%%%% cor 
\begin{cor}
\label{cor_expan_k}
For any monic polynomial $C_{k}(z)$ of degree $k$, $k\geq 0$, 
there exists a degree $k+1$ polynomial $Q_{k+1}(z)$, such that 
\begin{align}
\label{expan_k}
&
Q_{k+1}(z)P_{n+1}(z)+b_{n+1}(d_n-b_n)C_{k}(z)P_n(z;\alpha-1,\beta+1)  \\
&\quad
\in {\text{span}}\{P_{n+k+2}(z), \cdots, P_{n+2}(z)\}. \nonumber
\end{align}
\end{cor}

\begin{proof}
First, let us introduce the following formula 
\begin{align}
\label{expan_P_(-1,1)}
&
(d_{n}-b_{n})P_{n}(z;\alpha-1,\beta+1)=zP_{n}(z)-P_{n+1}(z), 
%=d_{n}P_{n}(z)-b_{n}zP_{n-1}(z). 
\end{align}
which is a straightforward consequence of (\ref{form_HR_0}). 
Then we can rewrite the second term on the left-hand side of (\ref{expan_k}) as 
\[
b_{n+1}(d_n-b_n)C_{k}(z)P_n(z;\alpha-1,\beta+1)
=b_{n+1}C_{k}(z)(zP_{n}(z)-P_{n+1}(z)). 
\]
In view of lemma \ref{lem_expan_P_n}, we assume that 
\[
b_{n+1}zC_{k}(z)=B_{k+1}(z)+c_{1}B_{k}(z)+\cdots+c_{k}B_{1}(z), 
\]
hence 
\begin{align*}
b_{n+1}zC_{k}(z)P_{n}(z)
&
=(B_{k+1}(z)+c_{1}B_{k}(z)+\cdots+c_{k}B_{1}(z))P_{n}(z) \\
&
=P_{n+k+2}(z)+c_1P_{n+k+1}(z)+\cdots+c_{k}P_{n+2}(z) \\
&\quad
-(D_{k+1}(z)+c_{1}D_{k}(z)+\cdots+c_{k}D_{1}(z))P_{n+1}(z). 
\end{align*}
Let $Q_{k+1}(z)$ be defined by 
\[
Q_{k+1}(z)=(D_{k+1}(z)+c_{1}D_{k}(z)+\cdots+c_{k}D_{1}(z))+b_{n+1}C_{k}(z)
\]
then the proof of (\ref{expan_k}) is completed. 

\end{proof}

\subsection{Expansion formulas related to HR polynomials}
From the recurrence relation (\ref{TTRR_HR}) one can deduce several expansion formulas related to 
HR polynomials and the ones with twisted parameters. 
First, by induction on (\ref{TTRR_HR}) one can obtain the following expansion formula. 
\begin{lemma}
\label{zHR}
For $n\geq 0$, the HR polynomials $P_n(z)$ satisfy the following formula: 
\begin{align}
\label{zHR}
&zP_{n}(z)=P_{n+1}(z)+\sum^{n}_{j=0}(-1)^{n-j}\prod^{n-j-1}_{l=0}b_{n-l}(d_j-b_j)P_{j}(z), 
\end{align}
where the coefficients $b_n$ and $d_n$ are given by (\ref{TTRR_Coe_HR}). 
\end{lemma}

One can also find the following connection between (\ref{TTRR_HR}) 
and the twisted HR polynomials $P_{n}(z;\alpha+1,\beta-1)$ through straightforward calculations using (\ref{HR_P}). 
\begin{lemma}
\label{twisted_HR}
The twisted HR polynomials $P_{}(z;\alpha+1,\beta-1)$ can be expanded in terms of HR polynomials as follows: 
\begin{align}
\label{twisted_HR_1}
&P_{n}(z;\alpha+1,\beta-1)=P_n(z)+b_nP_{n-1}(z), \\
\label{twisted_HR_2}
&zP_{n}(z;\alpha+1,\beta-1)=P_{n+1}(z)+d_nP_{n}(z). 
\end{align}
\end{lemma}

Moreover, by replacing $\alpha$ and $\beta$ to $\alpha-1$ and $\beta+1$ in both sides of (\ref{twisted_HR_1}) 
we obtain 
\begin{align}
\label{twisted_HR_3}
&P_{n}(z)=P_n(z;\alpha-1,\beta+1)+b_n^{\alpha-1,\beta+1}P_{n-1}(z;\alpha-1,\beta+1),  
%\\
%\label{twisted_HR_4}
%&zP_{n}(z)=P_{n+1}(z;\alpha-1,\beta+1)+d_n^{\alpha-1,\beta+1}P_{n}(z;\alpha-1,\beta+1), 
\end{align}
thus it follows from (\ref{HR_Q}) that
\begin{align}
\label{twisted_HR_Q_1}
Q_{n}(z)=Q_n(z;\alpha+1,\beta-1)+b_{n}^{\beta-1,\alpha+1}Q_{n-1}(z;\alpha+1,\beta-1), 
\end{align}
where we replaced $b_{n}(\beta-1,\alpha+1)$ to the new notation $b_{n}^{\beta-1,\alpha+1}$ for compactness of formulas. 
In fact, the three-term recurrence relation (\ref{TTRR_HR}) can also be derived by using (\ref{twisted_HR_Q_1}) 
and (\ref{twisted_HR_1}) in the following fashion: 
\begin{align*}
&
\langle w(z)z(P_n(z)+b_nP_{n-1}(z)), Q_m(z) \rangle
=\langle P_n(z)+b_nP_{n-1}(z), z^{-1}w(z^{-1})Q_m(z) \rangle \\
&
=\langle P_{n}(z;\alpha+1,\beta-1), 
-w(z^{-1};\alpha+1,\beta-1)(Q_m(z;\alpha+1,\beta-1)+b_{m}^{\beta-1,\alpha+1}Q_{m-1}(z;\alpha+1,\beta-1)) \rangle  \\
&
=0, \text{ if }m\leq n-1, 
\end{align*}
where the relation $z^{-1}w(z^{-1})=-w(z^{-1};\alpha+1,\beta-1)$ was used. 

By iterating (\ref{twisted_HR_3}) one can obtain the following formula: 
\begin{align}
\label{P(z;a-1,b+1)}
P_n(z;\alpha-1,\beta+1)=P_{n}(z)+\sum^{n}_{j=0}(-1)^{n-j}\prod^{n-j-1}_{l=0}b_{n-l}^{\alpha-1,\beta+1}P_{j}(z). 
\end{align}
Then it follows from (\ref{zHR}) and (\ref{P(z;a-1,b+1)}) that 
\[
zP_{n}(z;\alpha-1,\beta+1)=P_{n+1}(z)+\sum^{n}_{j=0}(-1)^{n-j}\prod^{n-j-1}_{l=0}b_{n-l}^{\alpha-1,\beta+1}
\left(\sum^{n-j}_{i=0}d_{n-i}^{\alpha-1,\beta+1}-\sum^{n+1-j}_{k=0}b_{n+1-k}^{\alpha-1,\beta+1}
\right)P_j(z). 
\]

Furthermore, by iterating (\ref{twisted_HR_1}) and (\ref{twisted_HR_Q_1}) 
one can obtain the following expansions.
\begin{lemma}
\label{lem_expan_PQ}
For $1\leq j\leq n$, it holds that 
\begin{align}
\label{twisted_HR_P_j}
&
P_n(z;\alpha+j,\beta-j)=P_n(z)+\sum^{j}_{l=1}C_{n,j}^{(l)}P_{n-l}(z), 
\\
\label{twisted_HR_Q_j}
&
Q_{n}(z)=Q_n(z;\alpha+j,\beta-j)+\sum^{j}_{l=1}E_{n,j}^{(l)}Q_{n-l}(z;\alpha+j,\beta-j), 
\end{align}
where the coefficients $C_{n,j}^{(l)}, E_{n,j}^{(l)}$, $l=1,\ldots,j$, are given by 
\begin{align}
\label{twisted_HR_P_j_coe}
&
C_{n,j}^{(1)}=\sum^{j-1}_{k=0}b_{n}^{\alpha+k,\beta-k}, \quad 
C_{n,j}^{(j)}=\prod^{j-1}_{k=0}b_{n-k}^{\alpha+j-1-k,\beta-j+1+k}, \\
\label{twisted_HR_Q_j_coe}
&
E_{n,j}^{(1)}=\sum^{j}_{k=1}b_{n}^{\beta-k,\alpha+k}, \quad 
E_{n,j}^{(j)}=\prod^{j}_{k=1}b_{n-k+1}^{\beta-k,\alpha+k}, 
\end{align}
and 
\begin{align}
\label{twisted_HR_P_j_coe2}
&
C_{n,j}^{(l)}=C_{n,j-1}^{(l)}+b_{n}^{\alpha+j-1,\beta-j+1}C_{n-1,j-1}^{(l-1)}, \quad l=2,\ldots, j-1, \\
\label{twisted_HR_Q_j_coe2}
&
E_{n,j}^{(l)}=E_{n,j-1}^{(l)}+b_{n-1}^{\beta-j,\alpha+j}E_{n,j-1}^{(l-1)}, \quad l=2,\ldots, j-1. 
\end{align}
\end{lemma}

Using these expansions one can prove the following lemma. 
\begin{lemma}
\label{lem_zp_expan}
There exist constants $a_{n,l_0}^{(0)}, \ldots, a_{n,l_0}^{(l_0+1)}$, such that for any positive integer $j$, 
$1\leq j\leq l_0+1$, it holds that 
\begin{align}
\label{p_span}
&
z^{j}\sum^{l_0+1}_{l=0}a_{n,l_0}^{(l)}P_{n-l}(z)
\in\textup{span}\{P_{n+j}(z), \cdots, P_{n+j-l_0-1}(z)\}, 
\end{align}
where $1\leq l_0\leq n-1$. 
Moreover, if $a_{n,l_0}^{(0)}=1$, then 
\begin{align}
\label{Coe_a}
a_{n,l_0}^{(l)}=C_{n,l_0+1}^{(l)}, \quad l=1, \ldots, l_0+1, 
\end{align}
where $C_{n,l_0+1}^{(l)}$, $ l=1, \ldots, l_0+1$, are defined by (\ref{twisted_HR_P_j_coe}) and (\ref{twisted_HR_P_j_coe2}). 
\end{lemma}

\begin{proof}
Firstly, for arbitrary constants $a_{n,l_0}^{(0)}, \ldots, a_{n,l_0}^{(l_0+1)}$ we have 
\begin{align*}
&
\langle w(z)z^{j}\sum^{l_0+1}_{l=0}a_{n,l_0}^{(l)}P_{n-l}(z), Q_m(z) \rangle
=\langle \sum^{l_0+1}_{l=0}a_{n,l_0}^{(l)}P_{n-l}(z), z^{-j}w(z^{-1})Q_m(z) \rangle. 
\end{align*}
By using the relation $w(z^{-1};\alpha+j,\beta-j)=(-1)^{j}z^{-j}w(z^{-1})$ and (\ref{twisted_HR_Q_j}) we have 
\begin{align}
\label{zjQ}
&
z^{-j}w(z^{-1})Q_n(z) \\
&
=(-1)^{j}w(z^{-1};\alpha+j,\beta-j)
(Q_n(z;\alpha+j,\beta-j)+\sum^{j}_{l=1}E_{n,j}^{(l)}Q_{n-l}(z;\alpha+j,\beta-j)).  \nonumber
\end{align}
Then according to (\ref{zjQ}), the relation (\ref{p_span}) holds if 
\begin{align*}
\sum^{l_0+1}_{l=0}a_{n,l_0}^{(l)}P_{n-l}(z)
\in\textup{span}\{P_{n}(z;\alpha+j,\beta-j), \cdots, P_{n+j-l_0-1}(z;\alpha+j,\beta-j)\} 
\end{align*}
is satisfied for $1\leq j\leq l_0+1$. This is because the above implies that 
\[
\langle w(z)z^{j}\sum^{l_0+1}_{l=0}a_{n,l_0}^{(l)}P_{n-l}(z), Q_m(z) \rangle=0, \quad m\leq n+j-l_0-2. 
\]
If $a_{n,l_0}^{(0)}=1$, then the coefficients $a_{n,l_0}^{(1)}, \ldots, a_{n,l_0}^{(l_0+1)}$ can be obtained 
from the equation
\begin{align}
\label{sol_a_eq}
\sum^{l_0+1}_{l=0}a_{n,l_0}^{(l)}P_{n-l}(z)=P_{n}(z;\alpha+l_0+1,\beta-l_0-1). 
\end{align}
Finally, it follows from (\ref{twisted_HR_P_j}) that (\ref{Coe_a}) was proved. 
%\begin{align}
%\label{sol_a}
%a_{n,l_0}^{(l)}=C_{l_0+1,l}, \quad l=1, \ldots, l_0+1. 
%\end{align}
\end{proof}

The following corollary is an immediate consequence of Lemma \ref{lem_zp_expan}. 
\begin{cor}
\label{cor_q_span}
Given any polynomial $q(z)$ of degree $l_0+1$ and $z\mid q(z)$, it holds that 
\begin{align}
\label{q_span}
&
q(z)\sum^{l_0+1}_{l=0}a_{n,l_0}^{(l)}P_{n-l}(z)
\in\textup{span}\{P_{n+l_0+1}(z), \cdots, P_{n-l_0}(z)\}, 
\end{align}
where $a_{n,l_0}^{(0)}=1$, and $a_{n,l_0}^{(1)}, \ldots, a_{n,l_0}^{(l_0+1)}$ are defined by (\ref{Coe_a}).
\end{cor}

%Here we list some examples of this type of recurrence relations.  \\
%\noindent
%\textbf{Examples.  }
%\begin{align*}
%&
%P_{n+2}(z)
%%=(z+b_{n+1}-d_{n+1})P_{n+1}(z)+b_{n+1}d_{n}P_{n}(z)-b_{n+1}b_{n}zP_{n-1}(z) \\
%%&
%=(z+b_{n+1}-d_{n+1})P_{n+1}(z)+b_{n+1}(d_{n}-b_{n})P_{n}(z;\alpha-1,\beta+1), 
%\end{align*}
%
%%Then we can go further as below 
%\begin{align*}
%&
%P_{n+3}(z)+(d_{n+2}+c_0)P_{n+2}(z)
%%=(z+b_{n+2}+c_{0})P_{n+2}(z)+b_{n+2}d_{n+1}P_{n+1}(z)-b_{n+2}b_{n+1}zP_{n}(z)  \\
%%&
%%=(z+b_{n+2}+c_{0})\left[(z+b_{n+1}-d_{n+1})P_{n+1}(z)+b_{n+1}d_{n}P_{n}(z)-b_{n+1}b_{n}zP_{n-1}(z)
%%\right] \\
%%&\quad
%%+b_{n+2}d_{n+1}P_{n+1}(z)-b_{n+2}b_{n+1}zP_{n}(z) \\
%%&
%%=\left[z^2+(b_{n+2}+b_{n+1}-d_{n+1}+c_0)z+(b_{n+1}-d_{n+1})c_0 \right]P_{n+1}(z)
%%+b_{n+2}b_{n+1}(P_{n+1}(z)-zP_{n}(z)) \\
%%&\quad
%%+b_{n+1}d_{n}(z+b_{n+2}+c_{0})P_{n}(z)-b_{n+1}b_{n}z(z+b_{n+2}+c_{0})P_{n-1}(z) \\
%%&
%%=\left[z^2+(b_{n+2}+b_{n+1}-d_{n+1}+c_0)z+(b_{n+1}-d_{n+1})c_0 \right]P_{n+1}(z)
%%+b_{n+1}d_{n}(z+c_0)P_{n}(z) \\
%%&\quad
%%-b_{n+1}b_{n}z(z+c_0)P_{n-1}(z) \\
%%&
%=\left[z^2+(b_{n+2}+b_{n+1}-d_{n+1}+c_0)z+(b_{n+1}-d_{n+1})c_0 \right]P_{n+1}(z) \\
%&\quad
%+b_{n+1}(d_{n}-b_{n})(z+c_0)P_{n}(z;\alpha-1,\beta+1),
%\end{align*}
%and 
%\begin{align*}
%&
%P_{n+4}(z)+(d_{n+2}+d_{n+3}-b_{n+3}+c_1)P_{n+3}(z)+(d_{n+2}(d_{n+2}-b_{n+3}+c_1)+c_0)P_{n+2}(z) \\
%&
%=[z^3+(b_{n+2}+b_{n+1}-d_{n+1}+c_1)z^2+(b_{n+2}(d_{n+2}-b_{n+3})+c_0+(b_{n+2}+b_{n+1}-d_{n+1})c_1)z \\
%&\quad
%+(b_{n+1}-d_{n+1})c_0
%]P_{n+1}(z) 
%+b_{n+1}(d_{n}-b_{n})(z^2+c_1z+c_0)P_{n}(z;\alpha-1,\beta+1). 
%\end{align*}
%Repeating this procedure finally leads us to 

\section{Exceptional Hendriksen-van Rossum polynomials}
In our previous work \cite{XHR}, we constructed four types of exceptional HR polynomials using 
a single-step Darboux transformation of a GEVP associated with differential operators. 
Here, we briefly review these results. 
According to the results revisited in Section 2, 
by choosing the eigenfunction $\psi(z)$ as an HR polynomial $P_n(z)$, 
the seed function $\phi(z)$ as a quasi-Laurent-polynomial eigenfunction, 
and the decoupling factor $\epsilon(z)$ as follows
i.e., 
\[
\psi(z)=\phi^{(1,n)}(z)=P_{n}(z), \quad 
\lambda=\theta_n=n, 
\]
\[
\phi(z)=\phi^{(j_0,l_0)}(z)=\xi_{j_0}(z)p^{(j_0)}_{l_0}(z), \quad 
\kappa=\theta^{(j_0)}_{l_0}, 
\]
\begin{align}
\label{eps}
\epsilon(z):=
\epsilon^{(j_0,l_0)}(z)=\dfrac{1}{Q^{(j_0)}(z)p^{(j_0)}_{l_0}(z)}, 
\end{align}
the Darboux transformed eigenfunction $\hat{\psi}(z)$ can be normalized into (\ref{hpsi_0}), 
where $j_0\in\{1,2,3,4\}$ and $l_0\in\mathbb{N}_{\geq 1}, n\in\mathbb{N}_{\geq 0}$. 
Moreover, the exceptional HR polynomials determined by (\ref{hphi_to_X1P})-(\ref{quasi_P_2}) can be 
rewritten in more compact forms: 
\begin{align}
\label{XP01}
&
P^{(1,l_0,n)}(z)
=nP_{l_0}(z)P_{n-1}(z;\alpha+1,\beta)-l_0P_{l_0-1}(z;\alpha+1,\beta)P_{n}(z), \\
\label{XP02}
&
P^{(2,l_0,n)}(z)
=(1-z)nP_{l_0}(z; -\beta,-\alpha)P_{n-1}(z;\alpha+1,\beta) \\
&\hspace{24mm}
+(l_0-\alpha-\beta)P_{l_0}(z; -\beta,-\alpha-1)P_{n}(z), \nonumber \\
\label{XP03}
&
P^{(3,l_0,n)}(z)
=\dfrac{(\beta)_{l_0}}{(\alpha+1)_{l_0}}\bigg[
nzP_{l_0}(z; \beta-1,\alpha+1)P_{n-1}(z; \alpha+1,\beta) \\
&\hspace{24mm}
+(\alpha+1)P_{l_0}(z; \beta-1,\alpha+2)P_{n}(z) \bigg],  \nonumber \\
\label{XP04}
&
P^{(4,l_0,n)}(z)
=\dfrac{(-\alpha)_{l_0}}{(-\beta+1)_{l_0}}\bigg[
z(z-1)nP_{l_0}(z; -\alpha-1,-\beta+1)P_{n-1}(z;\alpha+1,\beta) \\
&\hspace{24mm}
+(\alpha+1)P_{l_0+1}(z; -\alpha-2,-\beta+1)P_n(z) 
\bigg], \nonumber 
\end{align}
where the formulas (105)-(108) in \cite{XHR} and the following relations are used 
\begin{align*}
%\label{XP01_formulas_01}
&
(1-z)l_0P_{l_0-1}(z;1-\beta,-\alpha)+(\alpha+\beta)P_{l_0}(z;-\beta,-\alpha)
=-(l_0-\alpha-\beta)P_{l_0}(z;-\beta,-\alpha-1), \\
&
\dfrac{l_0}{\beta+l_0-1}P_{l_0-1}(z;\beta-1,\alpha+2)+P_{l_0}(z;\beta-1,\alpha+1)
=P_{l_0}(z;\beta-1,\alpha+2), \\
%\label{XP01_formulas_02}
&
\dfrac{(-\beta+1)_{l_0}}{-\alpha+l_0+1}(z-1)P_{l_0-1}(z;-\alpha-1,-\beta+2)
-(1-\beta-(\alpha+1)z)P_{l_0}(z;-\alpha-1,-\beta+1) \\
&\hspace{4mm}
=(\alpha+1)P_{l_0+1}(z;-\alpha-2,-\beta+1). 
\end{align*}

According to the discussions in the last part of section 2, the biorthogonal partner of $\hat{\psi}^{(j_0,l_0,n)}(z)$ 
is $\widehat{L}_2^{\ast}\widehat{\psi^{\ast}}(z)$. 
For $\psi^{\ast}(z)=\phi^{(1,n)\ast}(z)$, we have
%it follows from (\ref{backward_hchi}) and the relation (\ref{w_1_f1}) that 
\begin{align}
\label{hL2_hpsi}
&
\widehat{L}_2^{\ast}\widehat{\psi^{\ast}}(z)
=\widehat{L}_2^{\ast}(\theta^{(j_0)}_{l_0}\overline{\cG^{\ast}_1}-\mu^{(1)}_n\overline{\cG^{\ast}_2})
\tilde{\phi}_2(z^{-1})\phi^{(1,n)\ast}(z)
=\theta^{(j_0)}_{l_0}(\overline{\cG^{\ast}_1}-\overline{\cG^{\ast}_2})\tilde{\phi}_2(z^{-1})\phi^{(1,n)\ast}(z), 
%\\
%&
%=\epsilon(z^{-1})A_1(z^{-1})z^2\left(
%\left(\phi^{(1,n)\ast}(z)\right)'
%-\left(\dfrac{w'_1(z)}{w_1(z)}+\dfrac{\theta^{(j_0)}_{l_0} B_2(z^{-1})}{z^2A_1(z^{-1})}
%-\dfrac{(\phi^{(j_0,l_0)})'(z^{-1})}{z^2\phi^{(j_0,l_0)}(z^{-1})}\right)
%\phi^{(1,n)\ast}(z)
%\right) \\
%&
%=\dfrac{(z-1)w_1(z)}{Q^{(j_0)}(z^{-1})p^{(j_0)}_{l_0}(z^{-1})}
%\left(
%(\tilde{p}^{(1)}_n(z))'
%-\left(\ln{\left(z^{\theta^{(j_0)}_{l_0}}\phi^{(j_0,l_0)}(z^{-1})\right)}\right)'
%\tilde{p}^{(1)}_n(z)
%\right). \nonumber
\end{align}
which turns out to be 
\begin{equation}
\label{hL2psi}
\widehat{L}_2^{\ast}\widehat{\psi^{\ast}}(z)
=
\begin{cases}
\dfrac{(\beta)_{l_0}}{(1+\alpha)_{l_0}}
\dfrac{z^{-l_0}(z-1)w_1(z)}{Q^{(j_0)}(z^{-1})Q^{(j_0)}(z)(p^{(j_0)}_{l_0}(z^{-1}))^2}Q^{(j_0,l_0,n)}(z), & j_0=1, \\
\dfrac{(-\alpha)_{l_0}}{(-\beta+1)_{l_0}}
\dfrac{z^{-l_0}(z-1)w_1(z)}{Q^{(j_0)}(z^{-1})Q^{(j_0)}(z)(p^{(j_0)}_{l_0}(z^{-1}))^2}Q^{(j_0,l_0,n)}(z), & j_0=2, \\
\dfrac{(1+\alpha)_{l_0}}{(\beta)_{l_0}}
\dfrac{z^{-2l_0}(z-1)w_1(z)}{Q^{(j_0)}(z^{-1})Q^{(j_0)}(z)(\tilde{p}^{(j_0)}_{l_0}(z))^2}Q^{(j_0,l_0,n)}(z), & j_0=3, \\
\dfrac{(-\beta+1)_{l_0}}{(-\alpha)_{l_0}}
\dfrac{z^{-2l_0}(z-1)w_1(z)}{Q^{(j_0)}(z^{-1})Q^{(j_0)}(z)(\tilde{p}^{(j_0)}_{l_0}(z))^2}Q^{(j_0,l_0,n)}(z), & j_0=4. 
\end{cases}
\end{equation}
%%where the formula (\ref{formula_HR_1}) is used, and 
%where
%\begin{align*}
%&
%Q^{(j_0,l_0,n)}(z)=
%Q^{(j_0)}(z)\left(\tilde{p}^{(j_0)}_{l_0}(z)\tilde{p}^{(1)}_n(z))'-(\tilde{p}^{(j_0)}_{l_0}(z))'\tilde{p}^{(1)}_n(z)\right) 
%-P^{(j_0)}(z)\tilde{p}^{(j_0)}_{l_0}(z)\tilde{p}^{(1)}_n(z), \hspace{2mm} j_0=1,2, \\
%&
%Q^{(j_0,l_0,n)}(z)=
%z^{l_0}\left(Q^{(j_0)}(z)\left(\tilde{p}^{(j_0)}_{l_0}(z)\tilde{p}^{(1)}_n(z))'-(\tilde{p}^{(j_0)}_{l_0}(z))'\tilde{p}^{(1)}_n(z)\right) 
%-P^{(j_0)}(z)\tilde{p}^{(j_0)}_{l_0}(z)\tilde{p}^{(1)}_n(z)
%\right)
%\\
%&\hspace{26mm}
%-(1+\alpha-\beta)z^{l_0-1}Q^{(j_0)}(z)\tilde{p}^{(j_0)}_{l_0}(z)\tilde{p}^{(1)}_n(z), 
% \hspace{2mm} j_0=3,4. 
%\end{align*}

%Then the biorthogonal partner of $P^{(j_0,l_0,n)}(z)$ is $Q^{(j_0,l_0,n)}(z)$. 
%Similarily, the degrees of $Q^{(j_0,l_0,n)}(z)$, $j_0=1,2,3,4$, are given by
%\begin{align}
%\label{deg_XQ0}
%\deg{Q^{(j_0,l_0,n)}(z)}=n+l_0-\delta_{1,j}+\delta_{4,j}. 
%\end{align}
%%$n+l_0-1$, $n+l_0$, $n+l_0$ and $n+l_0+1$, respectively. 
%%The biorthogonal partner of $P^{(4,l_0,-l_0-1)}(z)=1$
%%is $Q^{(4,l_0,-l_0-1)}(z)=1$. 

The biorthogonal partner of $P^{(j_0,l_0,n)}(z)$ can also be written in an elegant form 
(similar to $Q_n(z)$). 
\begin{prop}[\cite{XHR}]
For $j_0\in\{1,2,3,4\}$, we have 
\begin{align}
\label{rela_X1PQ}
Q^{(j_0,l_0,n)}(z)=P^{(j_0,l_0,n)}(z; \beta-1,\alpha+1). 
\end{align}
%which means $Q^{(j_0,l_0,n)}(z)$ can be obtained from $P^{(j_0,l_0,n)}(z)$ through 
%$\alpha\rightarrow\beta-1, \beta\rightarrow\alpha+1$. 
\end{prop}
It is seen that the degrees of $Q^{(j_0,l_0,n)}(z)$, $j_0=1,2,3,4$, are given by
\begin{align}
\label{deg_XQ0}
\deg{Q^{(j_0,l_0,n)}(z)}=n+l_0-\delta_{1,j}+\delta_{4,j}. 
\end{align}

Consequently, the polynomial sequences $\{P^{(j_0,l_0,n)}(z)\}_{n\in\mathbb{Z}_{\geq 0}^{(j_0,l_0)}}$ and 
$\{Q^{(j_0,l_0,n)}(z)\}_{n\in\mathbb{Z}_{\geq 0}^{(j_0,l_0)}}$ form a biorthogonal system, where  
\begin{align}
\label{index_Z}
\mathbb{Z}_{\geq 0}^{(j_0,l_0)}=
\begin{cases}
\mathbb{Z}_{\geq 0}/\{l_0\}, & j_0=1, \\
\mathbb{Z}_{\geq 0}, & j_0=2,3, \\
\mathbb{Z}_{\geq 0}\cup\{-l_0-1\}, & j_0=4. 
\end{cases}
\end{align}

The weight function corresponding to these exceptional HR polynomials are 
\begin{align}
\label{XW1}
w^{(j_0,l_0)}(z)
%=\dfrac{(z^{-1}-1)w_j(z^{-1})}{Q_{j_0,1}(z)p^{(j_0)}_{l_0}(z)Q_{j_0,j}(z^{-1})\tilde{p}^{(j_0)}_{l_0}(z^{-1})}. 
=
\begin{cases}
\dfrac{(\beta)_{l_0}}{(1+\alpha)_{l_0}}
%\dfrac{z^{l_0}(z-1)w(z)}{Q^{(j_0)}(z)Q^{(j_0)}(z^{-1})(p^{(j_0)}_{l_0}(z))^2}, & j_0=1, \\
\dfrac{z^{l_0}(z-1)w(z)}{P^2_{l_0}(z)}, & j_0=1, \\
\dfrac{(-\alpha)_{l_0}}{(-\beta+1)_{l_0}}
%\dfrac{z^{l_0}(z-1)w(z)}{Q^{(j_0)}(z)Q^{(j_0)}(z^{-1})(p^{(j_0)}_{l_0}(z))^2}, & j_0=2, \\
\dfrac{z^{1+l_0}w(z)}{(1-z)P^2_{l_0}(z; -\beta,-\alpha)}, & j_0=2, \\
\dfrac{(1+\alpha)_{l_0}}{(\beta)_{l_0}}
%\dfrac{z^{1+l_0}(z-1)w(z)}{Q^{(j_0)}(z)Q^{(j_0)}(z^{-1})(\tilde{p}^{(j_0)}_{l_0}(z^{-1}))^2}, & j_0=3, \\
\dfrac{z^{l_0}(z-1)w(z)}{P^2_{l_0}(z; \beta-1,\alpha+1)}, & j_0=3, \\
\dfrac{(-\beta+1)_{l_0}}{(-\alpha)_{l_0}}
%\dfrac{z^{1+l_0}(z-1)w(z)}{Q^{(j_0)}(z)Q^{(j_0)}(z^{-1})(\tilde{p}^{(j_0)}_{l_0}(z^{-1}))^2}, & j_0=4. 
\dfrac{z^{1+l_0}w(z)}{(1-z)P^2_{l_0}(z; -\alpha-1,-\beta+1)}, & j_0=4.
\end{cases}
\end{align}
%where the $w(z)$ is defined by (\ref{w_HR}). 

\begin{theorem}
The biorthogonality relations for the four types of exceptional HR polynomials defined by (\ref{XP01})-(\ref{XP04}) 
and (\ref{rela_X1PQ}) are: 
\begin{align}
\label{Xbiorth1}
&
\dfrac{1}{2\pi}\int_{0}^{2\pi}w^{(j_0,l_0)}(e^{ix})P^{(j_0,l_0,n)}(e^{ix})Q^{(j_0,l_0,m)}(e^{-ix})dx 
=h^{(j_0,l_0)}_{n}\delta_{mn}, 
\end{align}
with
\begin{align}
\label{Xh1}
h^{(j_0,l_0)}_{n}
=
\begin{cases}
-(n+\beta)(n-l_0)h_n,  & j_0=1, \quad (n\neq l_0) \\
-(n+\beta)(n-l_0+\alpha+\beta)h_n, & j_0=2, \\
-(n+\beta)(n+l_0+1+\alpha+\beta)h_n, & j_0=3, \\
-(n+\beta)(n+l_0+1)h_n,  & j_0=4, \quad (n\neq -l_0-1)
\end{cases}
\end{align}
where the constants $h_n$ are defined by (\ref{w_HR}). 
\end{theorem}

Moreover, by using the relations (\ref{formula_HR_1}) and (\ref{formula_HR_2}), 
one finds that $(P^{(4,l_0,n)}(z))'$, $n=0,1,2,\ldots$, possess a common polynomial factor. 
\begin{lemma}
\label{lem_XP04_Cfactor}
For $n\geq 0$, it holds that 
\begin{align}
\label{XP04_Cfactor}
(P^{(4,l_0,n)}(z))'
%=(n+l_0+1)z^{l_0}p^{(4)}_{l_0}(z)\left[(\alpha+1)P_n(z)+(z-1)P'_n(z)\right]
=(n+l_0+1)(n+\alpha+1)z^{l_0}p^{(4)}_{l_0}(z)P_n(z;\alpha+1,\beta-1). 
\end{align}
\end{lemma}

\begin{proof}
According to the definition (\ref{hphi_to_X1P}) and (\ref{hpsi_0}), we have 
\[
P^{(4,l_0,n)}(z)=z^{l_0}\left[p^{(4)}_{l_0}(z)\big(Q^{(4)}(z)P'_n(z)-P^{(4)}(z)P_n(z)\big)
-(p^{(4)}_{l_0}(z))'Q^{(4)}(z)P_n(z)
\right], 
\]
then 
\begin{align}
&
(P^{(4,l_0,n)}(z))'=z^{l_0-1}p^{(4)}_{l_0}(z)\bigg[l_0\big(Q^{(4)}(z)P'_n(z)-P^{(4)}(z)P_n(z)\big)
+Q^{(4)}(z)P''_n(z) \nonumber \\
\label{XP04_d}
&\hspace{50mm}
+z\big((Q^{(4)}(z))'-P^{(4)}(z)\big)P'_n(z)-(P^{(4)}(z))'P_n(z)\bigg]  \\
&\quad
-z^{l_0}\bigg[\left(\frac{l_0}{z}Q^{(4)}(z)+(Q^{(4)}(z))'+P^{(4)}(z)\right)(p^{(4)}_{l_0}(z))'+Q^{(4)}(z)(p^{(4)}_{l_0}(z))''
\bigg]P_n(z). \nonumber \
\end{align}
It turns out that 
\[
\left(\frac{l_0}{z}Q^{(4)}(z)+(Q^{(4)}(z))'+P^{(4)}(z)\right)(p^{(4)}_{l_0}(z))'+Q^{(4)}(z)(p^{(4)}_{l_0}(z))''
=\frac{l_0(\beta-1)}{z}p^{(4)}_{l_0}(z). 
\]
In fact, from (\ref{PQ_1}), (\ref{PQ_2}) and lemma \ref{lem_quasi}, the above equation can be rewritten into 
\[
\left[(-\beta-l_0+(l_0+1-\alpha)z)\partial_z+z(z-1)\partial_z^2\right]P_{l_0}(z^{-1};-\beta,-\alpha)
=\frac{l_0(\beta-1)}{z}P_{l_0}(z^{-1};-\beta,-\alpha), 
\]
which is equivalent to (\ref{Deq_HR}) through the transformations $z\rightarrow z^{-1}$, 
$\alpha\rightarrow -\beta$ and $\beta\rightarrow -\alpha$. 
Therefore, (\ref{XP04_d}) can be rewritten as 
\[
(P^{(4,l_0,n)}(z))'=z^{l_0}p^{(4)}_{l_0}(z)\bigg[z(z-1)P''_n(z)+(-2+\beta-l_0+(3+\alpha+l_0)z)P'_n(z)
+(l_0+1)(\alpha+1)P_n(z)\bigg], 
\]
which can be further simplified into (\ref{XP04_Cfactor}) due to (\ref{Deq_HR}) and (\ref{opL2_P}). 
\end{proof}
Note that this property is unique to the case when $j_0=4$. 
And it will play an important role in determining the polynomial factor in the left-hand side of 
(\ref{RR-XHR}). 

%\begin{remark}
%By observing the expressions (\ref{XP01})-(\ref{XP04}) one can easily obtain 
%the leading coefficients of $P^{(j_0,l_0,n)}(z)$, thus 
%\begin{align}
%P^{(1,l_0,n)}(z)=&(n-l_0)z^{n+l_0-1}+\cdots \\
%P^{(2,l_0,n)}(z)=&-(n-l_0+\alpha+\beta)z^{n+l_0}+\cdots \\
%P^{(3,l_0,n)}(z)=&\dfrac{(n+1+\alpha)(\beta)_{l_0}}{(\alpha+1)_{l_0}}z^{n+l_0}+\cdots \\
%P^{(4,l_0,n)}(z)=&\dfrac{(n+1+\alpha)(-\alpha)_{l_0}}{(-\beta+1)_{l_0}}z^{n+l_0+1}+\cdots 
%\end{align}
%\end{remark}

%\subsection{Completeness of exceptional HR polynomials}
%
%Denote $\mathbb{R}[z]$ by the space of univariate polynomials with variable $z\in\mathbb{C}$ and real parameters. 
%Let 

%\newpage
\section{Proofs of Theorem \ref{thm_main_1} and Theorem \ref{thm_main_2}}
In this section, we will present the proofs of Theorem \ref{thm_main_1} and Theorem \ref{thm_main_2}, 
where the following operators played significant roles: 
\begin{align}
\label{forward}
\text{forward operator:} & \quad
\cF=\dfrac{1}{\varepsilon(z)}\left(\partial_z-\dfrac{\phi'(z)}{\phi(z)} I \right), \quad 
\cF\psi(z)=\widehat{\psi}(z), \\
\label{back_2}
\text{backward operator:} & \quad
\cM=\kappa\tilde{\phi}_2(z) ( \cG_1-\cG_2), \quad
\cM\widehat{{\psi}}(z) = (\lambda - \kappa) L_2{\psi(z)}. 
\end{align}
Here $\varepsilon(z)$ is defined by (\ref{eps}), $\tilde{\phi}_2(z)$ is defined by (\ref{tphi}), 
$\cG_1, \cG_2$ are defined by (\ref{opG}), 
the rightest relations of (\ref{forward}) and (\ref{back_2}) follow from 
(\ref{DT_f1}) and (\ref{backward_hpsi}), (\ref{opL}), respectively. 
Using (\ref{opG}) and (\ref{tphi}), the backward operator can be expressed as  
\begin{align}
%\label{back_1}
%\cM_{\lambda}
%&=\frac{\epsilon(z)}{\widetilde{\phi}_2(z)}
%\left[A_1(z)\partial_z+A_1(z)\left(\frac{\epsilon'(z)}{\epsilon(z)}+\frac{\phi'(z)}{\phi(z)}\right)+B_1(z)-\lambda B_2(z)
%\right], \\
\label{back_2_expr}
\cM
&=\epsilon(z)
\left[A_1(z)\partial_z+A_1(z)\left(\frac{\epsilon'(z)}{\epsilon(z)}+\frac{\phi'(z)}{\phi(z)}\right)+B_1(z)-\kappa B_2(z)
\right], 
\end{align}
where $A_1(z), B_1(z)$ and $B_2(z)$ are defined by (\ref{opL_ABC}). 

\begin{lemma}
\label{lem_Mpsi}
If $\psi(z)=P_n(z)$, $\phi(z)=\phi^{(j_0,l_0)}(z)$ and $\kappa=\theta_{l_0}^{(j_0)}$, 
then $\widehat{{\psi}}(z)=\widehat{\psi}^{(j_0,l_0,n)}(z)$, 
\begin{align}
%\cM_{n}\widehat{\psi}^{(j_0,l_0,n)}(z)
%&=(n-\theta_{l_0}^{(j_0)})P_n(z), \\
\label{Mpsi_0}
\cM \widehat{\psi}^{(j_0,l_0,n)}(z)
%=(\theta_n-\kappa)L_2 P_n(z)
&=\Xi_n^{(j_0,l_0)}P_n(z; \alpha+1,\beta-1), 
\end{align}
where $\Xi_n^{(j_0,l_0)}=-(n-\theta_{l_0}^{(j_0)})(n+\alpha+1)$. 
\end{lemma}

\begin{proof}
Substituting (\ref{opL2_P}) into the rightest relation of (\ref{back_2}), then one obtains (\ref{Mpsi_0}). 
\end{proof}

%for any $p(z)\in\{z^{-k}P(z)\mid P(z)\in\mathbb{R}[z]\}$, $k\in\mathbb{N}$, 
%if $\cM[p(z)]\in\mathbb{R}[z]$, then 
%$p(z)\in {\rm span}\{\widehat{\psi}^{(j_0,l_0,n)}(z)\}_{n\in\mathbb{Z}_{\geq 0}^{(j_0,l_0)}}$. 

%The idea of our proof is simple, we first show that there is some polynomial $q(z)$ such that 
%\[
%\cM[q(z)\widehat{\psi}^{(j_0,l_0,n)}(z)]=\sum^{n+l_0+1}_{j=0}c_{n,j}P_j(z; \alpha+1,\beta-1), 
%\]

\begin{remark}
From (\ref{back_2}) and (\ref{Mpsi_0}) one can easily derive 
\begin{align}
\label{MFp_0}
&
\langle w(z;\alpha+1,\beta-1)\cM \widehat{\psi}^{(j_0,l_0,n)}(z), Q_m(z;\alpha+1,\beta-1) \rangle 
=\Xi_n^{(j_0,l_0)}h_n^{\alpha+1,\beta-1}\delta_{nm},  
\end{align}
where the inner product $\langle, \rangle$ is defined by (\ref{inner_0}), $w(z)$ refers to the weight function (\ref{w_HR}), 
$h_n^{\alpha+1,\beta-1}$ is the norming constant $h_n$ with $\alpha$ and $\beta$ replaced by 
$\alpha+1$ and $\beta-1$. 

It follows from (\ref{decom_haL}) that 
\begin{align}
\label{haFM}
\overline{\cF^{\ast}\cM^{\ast}}
=\kappa\overline{\cF^{\ast}}(\overline{\cG_1^{\ast}}-\overline{\cG_2^{\ast}})\overline{\tilde{\phi}_2(z)}
=\overline{\cL_1^{\ast}}-\kappa\overline{\cL_2^{\ast}}. 
\end{align}
Therefore, under the assumption of lemma \ref{lem_Mpsi}, we have 
\begin{align}
\label{haFM_0}
\overline{\cF^{\ast}}[\overline{\cM^{\ast}}[\overline{w(z;\alpha+1,\beta-1)}Q_m(z;\alpha+1,\beta-1)]]
=-(m-\theta_{l_0}^{(j_0)})(m+\beta)\overline{w(z)}Q_m(z), 
\end{align}
where the following relation (which follows from remark 3.9 of \cite{XHR}) was used
\[
\overline{\cL_2^{\ast}}[\overline{w(z;\alpha+1,\beta-1)}Q_m(z;\alpha+1,\beta-1)]
=-(m+\beta)\overline{w(z)}Q_m(z). 
\]
If $j_0=1$ and $m=l_0$, then since $\theta_{l_0}^{(1)}=l_0$, it follows from (\ref{haFM_0}) that 
\[
\overline{\cF^{\ast}}[\overline{\cM^{\ast}}[\overline{w(z;\alpha+1,\beta-1)}Q_m(z;\alpha+1,\beta-1)]]=0. 
\]

The equation (\ref{haFM_0}) also implies the following equation
\begin{align}
\label{MFp_10}
\langle P_n(z), \overline{\cF^{\ast}}[\overline{\cM^{\ast}}[\overline{w(z;\alpha+1,\beta-1)}Q_{m}(z;\alpha+1,\beta-1)]] 
\rangle
=(\theta_{l_0}^{(j_0)}-n)(n+\beta)h_{n}\delta_{nm}, 
\end{align}
or, furthermore,  
\begin{align}
\label{MFp_1}
\langle \cM[\cF P_n(z)], \overline{w(z;\alpha+1,\beta-1)}Q_m(z;\alpha+1,\beta-1)
\rangle
=(\theta_{l_0}^{(j_0)}-n)(n+\beta)h_{n}\delta_{nm}, 
\end{align}
which is equivalent to (\ref{MFp_0}) since 
\begin{align}
\label{h_twisted_0}
h_n^{\alpha+1,\beta-1}=\frac{n+\beta}{n+\alpha+1}h_n. 
\end{align}
\end{remark}

\begin{remark}
By observing the eigenvalues $\theta_{l_0}^{(j_0)}$ defined in lemma \ref{lem_quasi}, it follows immediately that 
\begin{align}
\label{Mpsi_1}
&
\cM \widehat{\psi}^{(1,l_0,l_0)}(z)=0, \quad
\cM \widehat{\psi}^{(2,l_0,l_0-N)}(z;\alpha,N-\alpha)=0, \quad 0\leq N\leq l_0, \\
&
\cM \widehat{\psi}^{(4,l_0,-l_0-1)}(z)=0, 
\end{align}
where 
\[
\widehat{\psi}^{(1,l_0,l_0)}(z)=0, \quad 
\widehat{\psi}^{(2,l_0,l_0-N)}(z;\alpha,N-\alpha)=0, \quad
\widehat{\psi}^{(4,l_0,-l_0-1)}(z)=z^{-l_0}.
\]
\end{remark}

In what follows, we always assume that $\psi(z)$, $\phi(z)$ and $\kappa$ are defined as in lemma \ref{lem_Mpsi}, hence 
%the forward operator $\cF$ maps an HR polynomial $P_n(z)$ into an exceptional HR polynomial $P^{(j_0,l_0,n)}(z)$, 
%while 
the backward operator $\cM$ maps an exceptional HR polynomial $P^{(j_0,l_0,n)}(z)$ into an HR polynomial 
with twisted parameters. 

%The main idea of the proof of theorem \ref{thm_main_2} is as follows. 
%First, we show that for any Laurent polynomial whose only pole is $z=0$, 
%if the backward operator $\cM$ maps it into a polynomial, 
%then it must belong to the space spanned by the exceptional HR polynomials. 
%This is done in lemma \ref{lem_Mq}. 
%Then 

We have already known that the backward operator $\cM$ maps $\widehat{\psi}^{(j_0,l_0,n)}(z)$ to 
an HR polynomial with shifted parameters. 
In the following lemma, we will show the inverse, that is, 
if $\cM$ maps a Laurent polynomial $p(z)$ whose only pole is $z=0$ to a polynomial, 
then $p(z)$ must belong to the space spanned by $\{\widehat{\psi}^{(j_0,l_0,n)}(z)\}$. 

\begin{lemma}
\label{lem_Mq}
%Let $\mathbb{R}[z]$ be the polynomial space with real parameters. 
For any $p(z)\in\{z^{-k}P(z)\mid P(z)\in\mathbb{R}[z]\}$, $k\in\mathbb{N}$, then 
$\cM[p(z)]\in\mathbb{R}[z]$ if and only if 
%\textup{span}\{P_{n}(z;\alpha+1,\beta)\}_{n\in\mathbb{Z}_{\geq 0}}
\begin{align*}
p(z)\in {\rm span}\{\widehat{\psi}^{(j_0,l_0,n)}(z)\}_{n\in\mathbb{Z}_{\geq 0}^{(j_0,l_0)}}, 
\end{align*}
where 
\begin{align*}
\mathbb{Z}_{\geq 0}^{(j_0,l_0)}=
\begin{cases}
\mathbb{Z}_{\geq 0}/\{l_0\}, & j_0=1, \\
\mathbb{Z}_{\geq 0}, & j_0=2,3, \\
\mathbb{Z}_{\geq 0}\cup\{-l_0-1\}, & j_0=4. 
\end{cases}
\end{align*}
\end{lemma}

\begin{proof}
The sufficiency is obvious, so we only need to show that 
if $\cM[p(z)]\in\mathbb{R}[z]$, then 
$p(z)\in {\rm span}\{\widehat{\psi}^{(j_0,l_0,n)}(z)\}_{n\in\mathbb{Z}_{\geq 0}^{(j_0,l_0)}}$. 
We show it similarly with the proof of proposition 5.3 in \cite{X-Kra}. 
%By definition, the backward operator $\cM$ can be expressed as 
%\begin{align}
%\label{back_2_expr}
%\cM
%&=\epsilon(z)
%\left[A_1(z)\partial_z+A_1(z)\left(\frac{\epsilon'(z)}{\epsilon(z)}+\frac{\phi'(z)}{\phi(z)}\right)+B_1(z)-\kappa B_2(z)
%\right].
%\end{align}
Let $r(z)=\cM[p(z)]$, then we have
\[
Q^{(j_0)}(z)p^{(j_0)}_{l_0}(z)r(z)
=A_1(z)p'(z)+
\left(A_1(z)\left(\frac{P^{(j_0)}(z)-(Q^{(j_0)}(z))'}{Q^{(j_0)}(z)}\right)+B_1(z)-\kappa B_2(z)\right)p(z), 
\]
which leads to 
\begin{align}
\label{eq_pr_1}
p^{(1)}_{l_0}(z)r(z)
&=z(1-z)p'(z)
+(1-\beta-l_0+(l_0-\alpha-2)z)p(z), & \text{if } j_0=1, \\
\label{eq_pr_2}
(1-z)p^{(2)}_{l_0}(z)r(z)
&=z(1-z)p'(z)
+(1+\alpha-l_0+(l_0-\alpha-1)z)p(z),  &\text{if } j_0=2, \\
\label{eq_pr_3}
zp^{(3)}_{l_0}(z)r(z)
&=z(1-z)p'(z)
+(l_0-(l_0+\alpha+\beta+1)z)p(z), &\text{if } j_0=3, \\
\label{eq_pr_4}
z(1-z)p^{(4)}_{l_0}(z)r(z)
&=z(1-z)p'(z)
+(l_0-l_0 z)p(z), &\text{if } j_0=4. 
\end{align}
If $r(z)$ is a polynomial of degree $m$ which can be expanded in terms of the HR polynomials with shifted parameters 
as follows 
\begin{align*}
r(z)=\sum^{m}_{n=0, \hspace{1mm}n\in\mathbb{Z}_{\geq 0}^{(j_0,l_0)}}
\Xi_n^{(j_0,l_0)}C_n^{(j_0)}P_n(z; \alpha+1,\beta-1), 
\end{align*}
then one finds a particular solution to (\ref{eq_pr_1})-(\ref{eq_pr_4}) by using (\ref{Mpsi_0}): 
\[
p(z)=p_r^{(j_0,l_0)}(z)=\sum^{m}_{n=0, \hspace{1mm}n\in\mathbb{Z}_{\geq 0}^{(j_0,l_0)}}
C_n^{(j_0)}\widehat{\psi}^{(j_0,l_0,n)}(z)
\in {\rm span}\{\widehat{\psi}^{(j_0,l_0,n)}(z)\}_{n\in\mathbb{Z}_{\geq 0}^{(j_0,l_0)}}. 
\]
Note that if the expansion of $r(z)$ includes $P_{l_0}(z; \alpha+1,\beta-1)$ when $j_0=1$, 
or $P_{l_0-N}(z; \alpha+1, N-\alpha-1)$ when $j_0=2$ (where $0\leq N\leq l_0$), 
or $P_{-l_0-1}(z; \alpha+1, N-\alpha-1)$ (which corresponds to the added state, and it equals 0) when $j_0=4$, 
then the corresponding $p(x)$ does not exist and the assumption does not hold in view of (\ref{Mpsi_1}). 
So, by excluding $P_{l_0}(z; \alpha+1,\beta-1)$ when $j_0=1$ and $P_{l_0-N}(z; \alpha+1, N-\alpha-1)$ when $j_0=2$ 
and $P_{-l_0-1}(z; \alpha+1, N-\alpha-1)$ when $j_0=4$ 
from the expansion of $r(z)$, the equations (\ref{eq_pr_1})-(\ref{eq_pr_4}) can be solved by a general polynomials 
$r(z)$ of degree $m$ for any given $(j_0,l_0)$. 

Let $p(z)=p_0^{(j_0,l_0)}(z)+p_r^{(j_0,l_0)}(z)$ be a general solution to (\ref{eq_pr_1})-(\ref{eq_pr_4}), 
the equations of $p_0^{(j_0,l_0)}(z)$ then become a sequence of a homogeneous first-order differential equations 
with the right-hand sides equal to 0 in (\ref{eq_pr_1})-(\ref{eq_pr_4}). 
Thus $p_0^{(j_0,l_0)}(z)\in\text{Ker}\cM \cap \{z^{-k}P(z)\mid P(z)\in\mathbb{R}[z]\}$. 
It is easily seen that 
\begin{equation}
\label{Ker_M}
\text{Ker}\cM=
\begin{cases}
(1-z)^{-1-\alpha-\beta}z^{-1+\beta+l_0}, & j_0=1, \\
z^{-1-\alpha+l_0}, & j_0=2, \\
(1-z)^{-1-\alpha-\beta}z^{-l_0}, & j_0=3, \\
z^{-l_0}, & j_0=4, 
\end{cases}
\end{equation}
from which we obtain 
\[
p_0^{(1,l_0)}(z)=\widehat{\psi}^{(1,l_0,l_0)}(z), \quad 
p_0^{(2,l_0)}(z)=\widehat{\psi}^{(2,l_0,l_0-N)}(z;\alpha,N-\alpha), 
\]
\[
p_0^{(3,l_0)}(z)=0, \quad 
p_0^{(4,l_0)}(z)=z^{-l_0}=\widehat{\psi}^{(4,l_0,-l_0-1)}(z). 
\]
Therefore, we conclude that 
\[
p(z)\in {\rm span}\{\widehat{\psi}^{(j_0,l_0,n)}(z)\}_{n\in\mathbb{Z}_{\geq 0}^{(j_0,l_0)}}. 
\]
\end{proof}

\begin{lemma}
\label{lem_qpi}
Let $q^{(j_0)}_{l_0}(z)$ be a non-constant polynomial such that 
\begin{align}
\label{qpi_0}
q^{(j_0)}_{l_0}(z)\in
\begin{cases}
\textup{span}\{P^{(4,l_0,n)}(z;-\alpha-1,-\beta+1)\}_{n\in\mathbb{Z}_{\geq 0}^{(1,l_0)}}, & j_0=1, \\
\textup{span}\{P^{(4,l_0,n)}(z;\beta-1,\alpha+1)\}_{n\in\mathbb{Z}_{\geq 0}^{(2,l_0)}}, & j_0=2, \\
\textup{span}\{P^{(4,l_0,n)}(z;-\beta,-\alpha)\}_{n\in\mathbb{Z}_{\geq 0}^{(3,l_0)}}, & j_0=3, \\
\textup{span}\{P^{(4,l_0,n)}(z)\}_{n\in\mathbb{Z}_{\geq 0}^{(4,l_0)}}, & j_0=4, 
\end{cases}
\end{align}
then there exist some constants $c_{n,j}$, $j=0,\ldots,n+d_{j_0,l_0}$, such that 
\begin{align}
\label{expan_0}
\cM[q^{(j_0)}_{l_0}(z)\widehat{\psi}^{(j_0,l_0,n)}(z)]
=\sum^{n+d_{j_0,l_0}}_{j=0}c_{n,j}P_j(z; \alpha+1,\beta-1), 
\quad c_{n, d_{j_0,l_0}}\neq 0, 
\end{align}
where $d_{j_0,l_0}=\deg{q^{(j_0)}_{l_0}(z)}\geq l_0+1$. 
In particular, for the lowest degree $d_{j_0,l_0}=l_0+1$, $q^{(j_0)}_{l_0}(z)$ can be given by 
\begin{align}
\label{qpi}
q^{(j_0)}_{l_0}(z)=
\begin{cases}
C\int{p^{(j_0)}_{l_0}(z)}dz, & j_0=1,2, \\
C\int{z^{l_0}p^{(j_0)}_{l_0}(z)}dz, & j_0=3,4, 
\end{cases}
\end{align}
where $C$ is an arbitrary constant. 

\end{lemma}

\begin{proof}
Given a polynomial $q^{(j_0)}_{l_0}(z)$ and define 
\begin{align}
\label{pi_0}
\pi^{(j_0)}_{l_0}(z)=
\begin{cases}
\epsilon(z)A_1(z)(q^{(j_0)}_{l_0}(z))', & j_0=1,2, \\
\epsilon(z)A_1(z)(q^{(j_0)}_{l_0}(z))'z^{-l_0}, & j_0=3,4
\end{cases}
\end{align}
then it follows from (\ref{back_2_expr}) and (\ref{Mpsi_0}) that 
\begin{align}
\label{MqP_1}
\cM[q^{(j_0)}_{l_0}(z)\widehat{\psi}^{(j_0,l_0,n)}(z)]
&=\Xi_n^{(j_0,l_0)}q^{(j_0)}_{l_0}(z)P_n(z; \alpha+1,\beta-1) 
+\pi^{(j_0)}_{l_0}(z)P^{(j_0,l_0,n)}(z).
\end{align}
Since $\deg{P^{(4,l_0,n)}(z)}=n+l_0+1$, 
if $q^{(j_0)}_{l_0}(z)$ is defined by (\ref{qpi_0}), then it is easily seen that $\deg{q^{(j_0)}_{l_0}(z)}\geq l_0+1$. 
And it follows from lemma \ref{lem_XP04_Cfactor} 
and lemma \ref{lem_quasi} that $\pi^{(j_0)}_{l_0}(z)$ is a polynomial of degree $d_{j_0,l_0}-l_0+\delta_{1,j}-\delta_{4,j}$. 
Therefore, $\cM[q^{(j_0)}_{l_0}(z)\widehat{\psi}^{(j_0,l_0,n)}(z)]$ is a polynomial of degree $n+d_{j_0,l_0}$, 
which leads to (\ref{expan_0}). 

Moreover, if $q^{(j_0)}_{l_0}(z)$ is defined by (\ref{qpi}), 
and let $C=1$ without loss of generality, then 
\begin{align}
\label{pi}
\pi^{(j_0)}_{l_0}(z)=\pi^{(j_0)}(z)=\dfrac{A_1(z)}{Q^{(j_0)}(z)}=
\begin{cases}
z(1-z), & j_0=1, \\
z, & j_0=2, \\
1-z, & j_0=3, \\
-1, & j_0=4, 
\end{cases}
\end{align}
hence $\cM[q^{(j_0)}_{l_0}(z)\widehat{\psi}^{(j_0,l_0,n)}(z)]$ is a polynomial of degree $n+l_0+1$. 
\end{proof}

%\[
%q^{(j_0)}_{l_0}(z)=
%\begin{cases}
%\left(P_{l_0+1}(z; \alpha-1,\beta)-P_{l_0+1}(0; \alpha-1,\beta)
%%-\frac{(\beta)_{l_0+1}}{(\alpha)_{l_0+1}}
%\right)/(l_0+1), & j_0=1, \\ 
%\left(P_{l_0+1}(z; -\beta-1,-\alpha)-P_{l_0+1}(0; -\beta-1,-\alpha)
%%-\frac{(-\alpha)_{l_0+1}}{(-\beta)_{l_0+1}}
%\right)/(l_0+1), & j_0=2, \\ 
%\left(P_{l_0+1}(z; \beta-2,\alpha+1)-P_{l_0+1}(0; \beta-2,\alpha+1)
%%-\frac{\alpha+l_0+1}{\beta-1}
%\right)/(l_0+1), & j_0=3, \\ 
%\left(P_{l_0+1}(z; -\alpha-2,-\beta+1)-P_{l_0+1}(0; -\alpha-2,-\beta+1)
%%-\frac{\beta-l_0-1}{\alpha+1}
%\right)/(l_0+1), & j_0=4.
%\end{cases}
%\]
%For simplicity, we write 
%\[
%q^{(j_0)}_{l_0}(z)=
%\left(P_{l_0+1}(z; \alpha^{(j_0)},\beta^{(j_0)})-P_{l_0+1}(0; \alpha^{(j_0)},\beta^{(j_0)})
%\right)/(l_0+1), 
%\]
%where 
%$(\alpha^{(1)},\beta^{(1)})=(\alpha-1,\beta)$, 
%$(\alpha^{(2)},\beta^{(2)})=(-\beta-1,-\alpha)$, 
%$(\alpha^{(3)},\beta^{(3)})=(\beta-2,\alpha+1)$, 
%$(\alpha^{(4)},\beta^{(4)})=(-\alpha-2,-\beta+1)$. 

%It is obvious that the degree of $q^{(j_0)}_{l_0}(z)$ is $l_0+1$ for $j_0=1,2,3,4$, 
%hence $\cM[q^{(j_0)}_{l_0}(z)\widehat{\psi}^{(j_0,l_0,n)}(z)]$ is a polynomial of degree 
%$n+l_0+1$ for $j_0=1,2,3,4$. 

Let $C=1$, for the lowest degree $d_{j_0,l_0}=l_0+1$, 
it follows from (\ref{int_P}) that the polynomials $q^{(j_0)}_{l_0}(z)$ defined by (\ref{qpi}) can be chosen as follows: 
\begin{align}
\label{q_jl}
q^{(j_0)}_{l_0}(z)=
\begin{cases}
\dfrac{1}{(l_0+1)}\left(P_{l_0+1}(z; \alpha-1,\beta)-P_{l_0+1}(0; \alpha-1,\beta)\right), & j_0=1, \\ 
\dfrac{1}{(l_0+1)}\left(P_{l_0+1}(z; -\beta-1,-\alpha)-P_{l_0+1}(0; -\beta-1,-\alpha)\right), & j_0=2, \\ 
\dfrac{(\beta)_{l_0}}{(l_0+1)(\alpha+1)_{l_0}}\left(P_{l_0+1}(z; \beta-2,\alpha+1)-P_{l_0+1}(0; \beta-2,\alpha+1)
\right), & j_0=3, \\ 
\dfrac{(-\alpha)_{l_0}}{(l_0+1)(-\beta+1)_{l_0}}\left(P_{l_0+1}(z; -\alpha-2,-\beta+1)-P_{l_0+1}(0; -\alpha-2,-\beta+1)
\right), & j_0=4.
\end{cases}
\end{align}
and $z\mid q^{(j_0)}_{l_0}(z)$, $j_0=1,2,3,4$. 

Unless claimed in specific, in what follows, we always assume that $q^{(j_0)}_{l_0}(z)$ is defined by (\ref{q_jl}). 
In this case, there exist constants $c_{n,j}$ ($=c_{n,j}^{(j_0,l_0)}$), $j=0,1,\ldots,n+l_0+1$, such that 
\begin{align}
\label{expan_0}
\cM[q^{(j_0)}_{l_0}(z)\widehat{\psi}^{(j_0,l_0,n)}(z)]
=\sum^{n+l_0+1}_{j=0}c_{n,j}P_j(z; \alpha+1,\beta-1), 
\quad c_{n, n+l_0+1}\neq 0. 
\end{align}
Moreover, for $0\leq m\leq n+l_0+1$, 
it follows from (\ref{Biorth_HR}) and (\ref{expan_0}) that 
\begin{align}
\label{c_nm}
c_{n,m}=\frac{1}{h_m^{\alpha+1,\beta-1}}
\langle w(z; \alpha+1,\beta-1)\cM[q^{(j_0)}_{l_0}(z)\widehat{\psi}^{(j_0,l_0,n)}(z)], Q_m(z; \alpha+1,\beta-1)
\rangle, 
\end{align}
which can also be written as
\begin{align*}
c_{n,m}
%&=\frac{1}{h_m^{\alpha+1,\beta-1}}
%\langle q^{(j_0)}_{l_0}(z)\widehat{\psi}^{(j_0,l_0,n)}(z), 
%\overline{\cM^{\ast}}[\overline{w(z;\alpha+1,\beta-1)}Q_m(z; \alpha+1,\beta-1)] \rangle \\
&=\frac{1}{h_m^{\alpha+1,\beta-1}}
\langle P_n(z), 
\overline{\cF^{\ast}}[\overline{q^{(j_0)}_{l_0}(z)}
\overline{\cM^{\ast}}[\overline{w(z;\alpha+1,\beta-1)}Q_m(z; \alpha+1,\beta-1)]] \rangle. 
\end{align*}
%where the operators $\overline{\cM^{\ast}}$ and $\overline{\cF^{\ast}}$ are the adjoints of $\cM$ and $\cF$. 
By introducing the function 
\begin{align}
\label{mQ}
\tilde{Q}^{(j_0,l_0)}_m(z):=
\overline{\cF^{\ast}}[q^{(j_0)}_{l_0}(z^{-1})
\overline{\cM^{\ast}}[\overline{w(z; \alpha+1,\beta-1)}Q_m(z; \alpha+1,\beta-1)]]/\overline{w(z)}, 
\end{align}
the constant $c_{n,m}$ can be expressed in terms of a more compact inner product: 
\begin{align}
\label{c_nm_0}
c_{n,m}
%&=\frac{1}{h_m^{\alpha+1,\beta-1}}
%\langle P_n(z), w(z^{-1})\tilde{Q}^{(j_0,l_0)}_m(z) \rangle 
=\frac{1}{h_m^{\alpha+1,\beta-1}}
\langle w(z)P_n(z), \tilde{Q}^{(j_0,l_0)}_m(z) \rangle. 
\end{align}
%In particular, in the case $j_0=1$, it turns out that $c_{n,l_0}=0$ since $\widehat{\psi}^{(1,l_0,l_0)}(z)=0$. 

With the help of lemma \ref{lem_Mpsi}, lemma \ref{lem_Mq} and lemma \ref{lem_qpi}, we can prove that 
the exceptional HR polynomials $\{P^{(j_0,l_0,n)}(z)\}$ satisfy the recurrence relation of the shape 
\[
q^{(j_0)}_{l_0}(z)\sum^{k}_{l=0}a_n^{(j_0,l_0,l)}P^{(j_0,l_0,n)}(z)
=\sum^{n+l_0+1}_{j=m+1}b_n^{(j_0,l_0,j)}P^{(j_0,l_0,j)}(z)
\]
by showing that there exist coefficients $a_n^{(j_0,l_0,l)}, b_n^{(j_0,l_0,j)}$ such that, 
\[
\sum^{k}_{l=0}a_n^{(j_0,l_0,l)}\cM[q^{(j_0)}_{l_0}(z)\widehat{\psi}^{(j_0,l_0,n)}(z)]
=\sum^{n+l_0+1}_{j=m+1}b_n^{(j_0,l_0,j)}\cM[\widehat{\psi}^{(j_0,l_0,j)}(z)]. 
\]
%where $a_n^{(j_0,l_0,l)}, b_n^{(j_0,l_0,j)}$ are real coefficients. 

\begin{lemma}
\label{lem_XHR_RR_0}
Given a positive integer $k(\leq n)$, 
if there exist constants $a_0^{(j_0,l_0,n)}, \ldots, a_k^{(j_0,l_0,n)}$, which are not identically zero, 
such that 
\begin{align}
\label{cond_X1P_RR_type1_0}
\sum^{k}_{l=0}a_l^{(j_0,l_0,n)}c_{n-l,j}=0, \quad j=0, \ldots, m, \quad m\leq n+l_0-k+1, 
\end{align}
or $n+l_0-k+2\leq m\leq n+l_0+1$, 
\begin{align}
\label{cond_X1P_RR_type1_1}
\begin{cases}
\sum^{k}_{l=0}a_l^{(j_0,l_0,n)}c_{n-l,j}=0, & j=0, \ldots, n+l_0-k+1,  \\
\sum^{n+l_0+1-j}_{l=0}a_l^{(j_0,l_0,n)}c_{n-l,j}, & j=n+l_0-k+2, \ldots,m, 
\end{cases}
\end{align}
then the following recurrence relation holds: 
\begin{align}
\label{X1P_RR_type1_0}
q^{(j_0)}_{l_0}(z)\sum^{k}_{l=0}a_l^{(j_0,l_0,n)}P^{(j_0,l_0,n-l)}(z) 
=\sum^{n+l_0+1}_{j=m+1}b_j^{(j_0,l_0,n)}P^{(j_0,l_0,j)}(z),
\end{align}
where $b_j^{(j_0,l_0,n)}$, $j=m+1,\ldots,n+l_0+1$, are constants. 
\end{lemma}

\begin{proof}
First, for arbitrary constants $a_0^{(j_0,l_0,n)}, \ldots, a_k^{(j_0,l_0,n)}$, it follows from (\ref{expan_0}) that 
\begin{align*}
&
\cM[q^{(j_0)}_{l_0}(z)\sum^{k}_{l=0}a_l^{(j_0,l_0,n)}\widehat{\psi}^{(j_0,l_0,n-l)}(z)] 
=\sum^{k}_{l=0}a_l^{(j_0,l_0,n)}\sum^{n-l+l_0+1}_{j=0}c_{n-l,j}P_j(z; \alpha+1,\beta-1) \\
%&
%=\sum^{n+l_0-k+1}_{j=0}\sum^{k}_{l=0}a_n^{(l)}c_{n-l,j}P_j(z; \alpha+1,\beta-1)
%+\sum^{n+l_0+1}_{j=n+l_0-k+2}\sum^{n+l_0+1-j}_{l=0}a_n^{(l)}c_{n-l,j}P_j(z; \alpha+1,\beta-1) \\
&
=\sum^{n+l_0+1}_{j=0}\tilde{b}_j^{(j_0,l_0,n)}P_j(z; \alpha+1,\beta-1), 
\end{align*}
where 
\begin{align}
\label{eq_b_0}
\tilde{b}_j^{(j_0,l_0,n)}=
\begin{cases}
\sum^{k}_{l=0}a_l^{(j_0,l_0,n)}c_{n-l,j}, & j=0, \ldots, n+l_0-k+1, \\
\sum^{n+l_0+1-j}_{l=0}a_l^{(j_0,l_0,n)}c_{n-l,j}, & j=n+l_0-k+2, \ldots, n+l_0+1.
\end{cases}
\end{align}
Then the condition (\ref{cond_X1P_RR_type1_0}) or (\ref{cond_X1P_RR_type1_1}) implies that 
\begin{align}
\label{eq_b}
\tilde{b}_j^{(j_0,l_0,n)}=0, \quad j=0,\ldots,m, 
\end{align}
%always have non-zero solutions $\{a_n^{(j_0,l_0,0)}, \ldots, a_n^{(j_0,l_0,k)}\}$, 
hence 
\[
\cM[q^{(j_0)}_{l_0}(z)\sum^{k}_{l=0}a_l^{(j_0,l_0,n)}\widehat{\psi}^{(j_0,l_0,n-l)}(z)] 
=\sum^{n+l_0+1}_{j=m+1}\tilde{b}_n^{(j)}P_j(z; \alpha+1,\beta-1),
\]
which implies that there exist constants $b_{m+1}^{(j_0,l_0,n)}, \ldots$, $b_{n+l_0+1}^{(j_0,l_0,n)}$, such that 
(\ref{X1P_RR_type1_0}) holds in view of lemma \ref{lem_Mpsi}, lemma \ref{lem_Mq} and lemma \ref{lem_qpi}. 
\end{proof}

\noindent
{\textbf{Proof of Theorem 1.1}

It follows from Lemma \ref{lem_XHR_RR_0} that in either case 
(whether the condition (\ref{cond_X1P_RR_type1_0}) or (\ref{cond_X1P_RR_type1_1}) is satisfied or not) 
the recurrence relation 
\begin{align*}
q^{(j_0)}_{l_0}(z)\sum^{k}_{l=0}a_l^{(j_0,l_0,n)}P^{(j_0,l_0,n-l)}(z) 
=\sum^{n+l_0+1}_{j=0}b_j^{(j_0,l_0,n)}P^{(j_0,l_0,j)}(z),
\end{align*}
holds.

\noindent
{\textbf{Proof of Theorem 1.2}

Lemma \ref{lem_XHR_RR_0} combined with Proposition \ref{prop_Coe_a} complete the proof of 
Theorem 1.2. 

\begin{prop}
\label{prop_Coe_a}
For any integer $n\geq 2l_0+1$, 
there exist constants 
$a_l^{(j_0,l_0,n)}$, $l=0,1,\ldots,l_0+1$, 
such that 
\begin{align}
\label{cond_a_c}
\sum^{l_0+1}_{l=0} a_l^{(j_0,l_0,n)}c_{n-l,m}=0,  \quad 0\leq m\leq n-l_0-1.
\end{align}
More precisely, if we let $a_0^{(j_0,l_0,n)}=1$, then the following 
\begin{align}
\label{expre_Coe_a_12}
a_l^{(j_0,l_0,n)}&=C_{n,l_0+1}^{(l)}, \quad l=1,\ldots, l_0+1, \quad j_0=1,2; \\
\label{expre_Coe_a_34}
a_l^{(j_0,l_0,n)}&=C_{n,l_0+1}^{(l)}(\alpha+1,\beta-1)\frac{\Xi_{n}^{(j_0,l_0)}}{\Xi_{n-l}^{(j_0,l_0)}}, 
\quad l=1,\ldots, l_0+1, \quad j_0=3,4,
\end{align}
is a set of solutions to (\ref{cond_a_c}), 
where $C_{n,l_0+1}^{(l)}$, $ l=1, \ldots, l_0+1$, are defined by (\ref{twisted_HR_P_j_coe}) and (\ref{twisted_HR_P_j_coe2}), the definition of $\Xi_{n}^{(j_0,l_0)}$ refers to lemma \ref{lem_Mpsi}. 
\end{prop}

In what follows, we will prove this proposition through two different approaches, 
each of which solves the coefficients $a_l^{(j_0,l_0,n)}$, $l=0,1,\ldots,l_0+1$, in (\ref{cond_a_c}) partially. 

\noindent
{\textbf{Proof of (\ref{expre_Coe_a_34}) in Proposition \ref{prop_Coe_a}}: }

On one hand, if we use the $c_{n,m}$'s defined by (\ref{c_nm}), then it suffices to show that 
\[
\sum^{l_0+1}_{l=0}a_l^{(j_0,l_0,n)}\cM[q^{(j_0)}_{l_0}(z)\widehat{\psi}^{(j_0,l_0,n-l)}(z)]\in
\text{span}\{P_{n+l_0+1}(z;\alpha+1,\beta-1), \ldots, P_{n-l_0}(z;\alpha+1,\beta-1)\}. 
\]
It turns out that there exist constants 
$\tilde{a}_0^{(j_0,l_0,n)} (=1), \tilde{a}_1^{(j_0,l_0,n)}, \ldots, \tilde{a}_{l_0+1}^{(j_0,l_0,n)}$ such that 
\begin{align*}
&
q^{(j_0)}_{l_0}(z)\sum^{l_0+1}_{l=0} \tilde{a}_l^{(j_0,l_0,n)}\Xi_{n-l}^{(j_0,l_0)}P_{n-l}(z;\alpha+1,\beta-1) \\
%=\sum^{n+l_0+1}_{j=n-l_0}\tilde{c}^{(j_0,l_0,n)}_{j}P_{j}(z;\alpha+1,\beta-1).
&\quad
\in\textup{span}\{P_{n+l_0+1}(z;\alpha+1,\beta-1), \cdots, P_{n-l_0}(z;\alpha+1,\beta-1)\}. 
\end{align*}
In fact, since $z\mid q^{(j_0)}_{l_0}(z)$ and $\deg{q^{(j_0)}_{l_0}(z)}=l_0+1$, 
it follows from Corollary \ref{cor_q_span} that
\[
\tilde{a}_l^{(j_0,l_0,n)}(\alpha-1,\beta+1)\Xi_{n-l}^{(j_0,l_0)}(\alpha-1,\beta+1)
=C_{n,l_0+1}^{(l)}\Xi_{n}^{(j_0,l_0)}(\alpha-1,\beta+1), \quad l=1,\ldots, l_0+1, 
\]
hence 
\begin{align}
\label{Coe_a_34}
\tilde{a}_l^{(j_0,l_0,n)}=C_{n,l_0+1}^{(l)}(\alpha+1,\beta-1)\frac{\Xi_{n}^{(j_0,l_0)}}{\Xi_{n-l}^{(j_0,l_0)}}, 
\quad l=0,\ldots, l_0+1, \quad j_0=1,2,3,4. 
\end{align}
Moreover, it turns out that for $j_0=3$ or $4$, the constants defined by (\ref{Coe_a_34}) ensure that 
\[
\pi^{(j_0)}(z)\sum^{l_0+1}_{l=0}\tilde{a}_l^{(j_0,l_0,n)}P^{(j_0,l_0,n-l)}(z)
%=\sum^{n+l_0+1}_{j=n-l_0}\tilde{\tilde{c}}^{(j_0,l_0,n)}_{j}P_{j}(z;\alpha+1,\beta-1),
\in\text{span}\{P_{n+l_0+1}(z;\alpha+1,\beta-1), \ldots, P_{n-l_0}(z;\alpha+1,\beta-1)\}, 
\]
%\[
%q(z)\sum^{l_0+1}_{l=0}a_{n,l_0}^{(l)}(\alpha+1,\beta-1)P_{n}(z;\alpha+1,\beta-1)
%\in\textup{span}\{P_{n+l_0+1}(z;\alpha+1,\beta-1), \cdots, P_{n-l_0}(z;\alpha+1,\beta-1)\}, 
%\]
while this does not hold for the case when $j_0=1$ or $2$. 
In fact, this statement follows from the following lemma. 

\begin{lemma}
For $j_0=3$ or $4$, there exist a polynomial $q^{(j_0)}_{n,l_0+1}(z)$ of degree $l_0+1$ 
and satisfying $z\mid q^{(j_0)}_{n,l_0+1}(z)$, and constants $e^{(j_0,l_0)}_{j}, j=n+1,\ldots,n+l_0+1$, such that 
\begin{align}
\label{cond_pi_P_0}
&
\pi^{(j_0)}(z)P^{(j_0,l_0,n)}(z) \\
&\quad
=q^{(j_0)}_{n,l_0+1}(z)\Xi_{n}^{(j_0,l_0)}P_{n}(z;\alpha+1,\beta-1)
+\sum^{n+l_0+1}_{j=n+1}e^{(j_0,l_0)}_{j}P_{j}(z;\alpha+1,\beta-1). \nonumber
\end{align}
% but not for $j_0=1,2$. 
\end{lemma}

\begin{proof}
First, let us consider the case $j_0=3$. It follows from (\ref{form_HR_1}), (\ref{form_HR_2}) 
and (\ref{twisted_HR_1}) that 
\begin{align*}
&
\pi^{(3)}(z)P^{(3,l_0,n)}(z)
%=\dfrac{(\beta)_{l_0}}{(\alpha+1)_{l_0}}
\propto
\bigg[
z(1-z)P_{l_0}(z; \beta-1,\alpha+1)P'_{n}(z) 
+(\alpha+1)(1-z)P_{l_0}(z; \beta-1,\alpha+2)P_{n}(z) \bigg] \\
&
=\left[(\alpha+1)(1-z)P_{l_0}(z;\beta-1,\alpha+2)-nzP_{l_0}(z;\beta-1,\alpha+1)\right]P_n(z) \\
&\quad
+\frac{n(n+\alpha+\beta)}{n+\alpha}zP_{l_0}(z;\beta-1,\alpha+1)P_{n-1}(z)
\\
&
=\bigg[(\alpha+1)\bigg(P_{l_0}(z;\beta-1,\alpha+1)+\frac{l_0(1-z)}{l_0-1+\beta}P_{l_0-1}(z;\beta-1,\alpha+2)\bigg) \\
&\quad\quad
-(n+\alpha+1)zP_{l_0}(z;\beta-1,\alpha+1)\bigg]P_n(z)
+\frac{n(n+\alpha+\beta)}{n+\alpha}zP_{l_0}(z;\beta-1,\alpha+1)P_{n-1}(z)
\\
&
=\bigg[(\alpha+1)\bigg(P_{l_0}(z;\beta-1,\alpha+1)+\frac{l_0(1-z)}{l_0-1+\beta}P_{l_0-1}(z;\beta-1,\alpha+2)\bigg)\bigg]
P_n(z) \\
&\quad
-(n+\alpha+1)zP_{l_0}(z;\beta-1,\alpha+1)(P_{n}(z)+b_nP_{n-1}(z))
\\
&
=
%\frac{zP_{l_0}(z;\beta-1,\alpha+1)}{n-\theta^{(3)}_{l_0}}\Xi^{(3,l_0)}_{n}P_{n}(z;\alpha+1,\beta-1)
-(n+\alpha+1)zP_{l_0}(z;\beta-1,\alpha+1)P_{n}(z;\alpha+1,\beta-1)
+\frac{(\alpha+1)(\beta-1)}{l_0-1+\beta}P_{l_0}(z;\beta-2,\alpha+2)P_{n}(z)  \\
%&
%\propto
%-\frac{l_0-1+\beta}{n+2+\alpha}zP_{l_0}(z;\beta-1,\alpha+1)P_{n}(z;\alpha+1,\beta-1)
%+\frac{(\alpha+1)(\beta-1)}{(n+1+\alpha)(n+2+\alpha)}P_{l_0}(z;\beta-2,\alpha+2)P_{n}(z) \\
&
=\left[-(n+\alpha+1)zP_{l_0}(z;\beta-1,\alpha+1)
+\frac{(\alpha+1)(\beta-1)}{l_0-1+\beta}P_{l_0}(z;\beta-2,\alpha+2)-\tilde{P}(z)
\right]P_{n}(z;\alpha+1,\beta-1) \\
&\quad
-\frac{(\alpha+1)(\beta-1)}{l_0-1+\beta}b_{n}P_{l_0}(z;\beta-2,\alpha+2)P_{n-1}(z)+\tilde{P}(z)P_{n}(z;\alpha+1,\beta-1), 
\end{align*}
where the relations 
%{\color{red}
\[
P_{l_0}(z;\beta-1,\alpha+2)=P_{l_0}(z;\beta-1,\alpha+1)+\frac{l_0}{l_0-1+\beta}P_{l_0-1}(z;\beta-1,\alpha+2), 
\]
\[
zP_{l_0-1}(z;\beta-1,\alpha+2)=P_{l_0}(z;\beta-1,\alpha+2)+\frac{\beta-1}{l_0}\left(
P_{l_0}(z,\beta-1,\alpha+2)-P_{l_0}(z,\beta-2,\alpha+2)\right), 
\]
%}
were also used. 
%and 
%\[
%zP_{l_0}(z;\beta-1,\alpha+1)=P_{l_0+1}(z;\beta-1,\alpha+1)
%-\frac{(\alpha+1)(\beta-1)}{(l_0-1+\beta)(l_0+\beta)}P_{l_0}(z;\beta-2,\alpha+2)
%\]
%which follows from (\ref{form_HR_0}) by replacing $(\alpha,\beta)$ to $(\beta-1,\alpha+1)$ was used. 
Therefore, to prove (\ref{cond_pi_P_0}), it suffices to show that there exists a polynomial $\tilde{P}(z)$ 
of degree $l_0+1$, such that 
\begin{align*}
&
-\frac{(\alpha+1)(\beta-1)}{l_0-1+\beta}b_nP_{l_0}(z;\beta-2,\alpha+2)P_{n-1}(z)+\tilde{P}(z)P_{n}(z;\alpha+1,\beta-1) \\
&\quad
\in\text{span}\{P_{n+l_0+1}(z;\alpha+1,\beta-1), \ldots, P_{n+1}(z;\alpha+1,\beta-1)\}. 
\end{align*}
In fact, by multiplying the factor $\frac{l_0-1+\beta}{(n+1+\alpha)(n+2+\alpha)}$ to the left-hand side and then 
do the following shifts $\alpha \rightarrow \alpha-1, \beta \rightarrow \beta+1, n\rightarrow n+1$, 
it is equivalent to showing that 
\begin{align*}
&
b_{n+1}(d_n-b_n)P_{l_0}(z;\beta-1,\alpha+1)P_{n}(z;\alpha-1,\beta+1)+\tilde{\tilde{P}}(z)P_{n+1}(z) \\
&\quad
\in\text{span}\{P_{n+l_0+2}(z), \ldots, P_{n+2}(z)\}, 
\end{align*}
where $\deg\tilde{\tilde{P}}(z)=l_0+1$. 
The above relation can be implied by corollary \ref{cor_expan_k}.

Similarly, for $j_0=4$, it turns out that 
\begin{align*}
&
\pi^{(4)}(z)P^{(4,l_0,n)}(z)
%=\dfrac{(-\alpha)_{l_0}}{(1-\beta)_{l_0}}
\propto
\bigg[
z(1-z)P_{l_0}(z; -\alpha-1,1-\beta)P'_{n}(z) 
-(\alpha+1)P_{l_0+1}(z; -\alpha-2,1-\beta)P_n(z) 
\bigg] \\
&
%=\dfrac{(-\alpha)_{l_0}}{(1-\beta)_{l_0}}
=
-(n+\alpha+1)zP_{l_0}(z;-\alpha-1,1-\beta)P_{n}(z;\alpha+1,\beta-1)
+\frac{(\alpha+1)(\beta-1)}{l_0-1-\alpha}P_{l_0}(z;-\alpha-2,2-\beta)P_{n}(z) \\
&
=\bigg[
-(n+\alpha+1)zP_{l_0}(z;-\alpha-1,1-\beta)+\frac{(\alpha+1)(\beta-1)}{l_0-1-\alpha}P_{l_0}(z;-\alpha-2,2-\beta)-\tilde{P}(z)
\bigg]P_{n}(z;\alpha+1,\beta-1) \\
&\quad
-\frac{(\alpha+1)(\beta-1)}{l_0-1-\alpha}b_nP_{l_0}(z;-\alpha-2,2-\beta)P_{n-1}(z)
+\tilde{P}(z)P_{n}(z;\alpha+1,\beta-1), 
\end{align*}
where $\tilde{P}(z)$ is a polynomial of degree $l_0+1$ 
and it does not necessarily coincide with the one in the case $j_0=3$. 
Again, after multiplying the factor $\frac{l_0-1-\alpha}{(n+1+\alpha)(n+2+\alpha)}$ to the left-hand side and then 
do the shifts $\alpha \rightarrow \alpha-1, \beta \rightarrow \beta+1, n\rightarrow n+1$, 
it suffices to show that 
\begin{align*}
&
b_{n+1}(d_n-b_n)P_{l_0}(z;-\alpha-1,1-\beta)P_{n}(z;\alpha-1,\beta+1)+\tilde{\tilde{P}}(z)P_{n+1}(z) \\
&\quad
\in\text{span}\{P_{n+l_0+2}(z), \ldots, P_{n+2}(z)\}, 
\end{align*}
where $\deg\tilde{\tilde{P}}(z)=l_0+1$. 
Again, this relation can be implied by corollary \ref{cor_expan_k}. 
In conclusion, the proof of (\ref{cond_pi_P_0}) is completed. 
\end{proof}

%\[
%\pi^{(1)}(z)P^{(1,l_0,n)}(z)
%=z(1-z)(P_{l_0}(z)P'_{n}(z)-P'_{l_0}(z)P_{n}(z)),
%\]
%\[
%\pi^{(2)}(z)P^{(2,l_0,n)}(z)
%=z(1-z)P_{l_0}(z; -\beta,-\alpha)P'_{n}(z) 
%+z(l_0-\alpha-\beta)P_{l_0}(z; -\beta,-\alpha-1)P_{n}(z)
%\]
%
%where the following formula was used 
%\[
%P'_{n+1}(z)=(n+1)\sum^{n}_{k=0}\frac{(k+1)_{n-k}}{(\alpha+k+2)_{n-k}}P_{k}(z;\alpha+1,\beta-1)
%\]

%\[
%\pi^{(j_0)}(z)\sum^{l_0+1}_{l=0}\tilde{a}_l^{(j_0,l_0,n)}P^{(j_0,l_0,n-l)}(z)
%=\pi^{(j_0)}(z)\sum^{l_0+1}_{l=0}
%C_{n,l_0+1}^{(l)}(\alpha+1,\beta-1)\frac{\Xi_{n}^{(j_0,l_0)}}{\Xi_{n-l}^{(j_0,l_0)}}
%P^{(j_0,l_0,n-l)}(z)
%\]

%
%\[
%zP_{n}(z;\alpha+1,\beta-1)=P_{n+1}(z)+d_{n}P_{n}(z)
%\]
%\[
%z^2P_{n}(z;\alpha+1,\beta-1)=z(P_{n+1}(z)+d_{n}P_{n}(z))
%\]
%\[
%=P_{n+2}(z;\alpha-1,\beta+1)+d_{n+1}^{\alpha-1,\beta+1}P_{n+1}(z;\alpha-1,\beta+1)
%+d_{n}(P_{n+1}(z;\alpha-1,\beta+1)+d_{n-1}^{\alpha-1,\beta+1}P_{n}(z;\alpha-1,\beta+1))
%\]

Therefore, we have finished the proof for (\ref{expre_Coe_a_34}) in view of (\ref{MqP_1}). 

\noindent
{\textbf{Proof of (\ref{expre_Coe_a_12}) in Proposition \ref{prop_Coe_a}}: }

On the other hand, if we use the $c_{n,m}$'s defined by (\ref{c_nm_0}), we should first figure out what kind of function 
$\tilde{Q}^{(j_0,l_0)}_m(z)$ is, and then do similar discussions as above. 
\begin{lemma}
\label{mQ_span}
For $l_0\in\mathbb{Z}_{\geq 1}$, $\tilde{Q}^{(j_0,l_0)}_m(z)$ is a Laurent polynomial, and more precisely, 
\begin{align}
\label{mQ_span_12}
\tilde{Q}^{(j_0,l_0)}_m(z)\in\text{span}\{z^{m-1}, \ldots, z^{-l_0-1}\}, \quad \text{if } j_0=1,2;  \\
\label{mQ_span_34}
\tilde{Q}^{(j_0,l_0)}_m(z)\in\text{span}\{z^{m-1}, \ldots, z^{-l_0-2}\}, \quad \text{if } j_0=3,4. 
\end{align}
\end{lemma}

\begin{proof}
By definition, we have 
\begin{align*}
&\overline{\cF^{\ast}}=z\partial_z\dfrac{z}{\epsilon(z^{-1})} I-\dfrac{\phi'(z^{-1})}{\epsilon(z^{-1})\phi(z^{-1})}I, \\
&\overline{\cM^{\ast}}=z\partial_z zA_1(z^{-1})\epsilon(z^{-1})+\epsilon(z^{-1})\left(B_1(z^{-1})-\kappa B_2(z^{-1})
+A_1(z^{-1})\left(\dfrac{\epsilon'(z^{-1})}{\epsilon(z^{-1})}+\dfrac{\phi'(z^{-1})}{\phi(z^{-1})}\right)\right), 
\end{align*}
where $\epsilon(z)$ is defined by (\ref{eps}), $\phi(z)=\phi^{(j_0)}_{l_0}(z)$ (see lemma \ref{lem_quasi}), 
$A_1(z), B_1(z), B_2(z)$ are defined by (\ref{opL_ABC}). 
From the expression of $\overline{\cF^{\ast}}$, one can immediately obtain the relation
\[
\overline{\cF^{\ast}}q^{(j_0)}_{l_0}(z^{-1})=q^{(j_0)}_{l_0}(z^{-1})\overline{\cF^{\ast}}
-\frac{(q^{(j_0)}_{l_0})'(z^{-1})}{\epsilon(z^{-1})}I, 
\]
thus we have 
\[
\overline{\cF^{\ast}}q^{(j_0)}_{l_0}(z^{-1})\overline{\cM^{\ast}}
=q^{(j_0)}_{l_0}(z^{-1})\overline{\cF^{\ast}}\overline{\cM^{\ast}}
-\frac{(q^{(j_0)}_{l_0})'(z^{-1})}{\epsilon(z^{-1})}\overline{\cM^{\ast}}. 
\]
Moreover, for $\psi^{\ast}(z)=\phi^{(1,m)\ast}(z)$, it holds that (see (\ref{hL2_hpsi}))
\[
\overline{\cM^{\ast}}[w(z^{-1};\alpha+1,\beta-1)Q_m(z;\alpha+1,\beta-1)]=
\frac{w(z^{-1};\alpha+1,\beta-1)}{w_1(z)}
\widehat{L}_2^{\ast}\widehat{\psi^{\ast}}(z),
\]
where the relation $|w_1(z)|=|w(z^{-1};\alpha+1,\beta-1)|$ was used. 
Then, from (\ref{haFM_0}) we obtain 
\begin{align*}
&
\overline{\cF^{\ast}}[q^{(j_0)}_{l_0}(z^{-1})
\overline{\cM^{\ast}}[\overline{w(z; \alpha+1,\beta-1)}Q_m(z; \alpha+1,\beta-1)]] \\
&\quad
=(\theta^{(j_0)}_{l_0}-m)(m+\beta)q^{(j_0)}_{l_0}(z^{-1})w(z^{-1})Q_m(z)
-\frac{(q^{(j_0)}_{l_0})'(z^{-1})}{\epsilon(z^{-1})}\frac{w(z^{-1};\alpha+1,\beta-1)}{w_1(z)}\widehat{L}_2^{\ast}\widehat{\psi^{\ast}}(z). 
\end{align*}
By using (\ref{hL2psi}) and the following relations
\begin{align}
\label{mQ_exp}
\frac{p^{(3)}_{l_0}(z^{-1})}{\tilde{p}^{(3)}_{l_0}(z)}=\frac{(\beta)_{l_0}}{(\alpha+1)_{l_0}}z^{l_0}, \hspace{2mm}
\frac{p^{(4)}_{l_0}(z^{-1})}{\tilde{p}^{(4)}_{l_0}(z)}=\frac{(-\alpha)_{l_0}}{(-\beta+1)_{l_0}}z^{l_0}, \hspace{2mm}
\frac{w(z^{-1};\alpha+1,\beta-1)}{w(z^{-1})}=-\frac{1}{z}, 
\end{align}
the expression of $\tilde{Q}^{(j_0,l_0)}_m(z)$ can be derived as: 
\begin{align}
\label{mQ_exp}
\tilde{Q}^{(j_0,l_0)}_m(z)
=\theta_{m}^{(j_0,l_0)}q^{(j_0)}_{l_0}(z^{-1})Q_m(z) 
+C^{(j_0,l_0)}\frac{z^{-l_0-1}(z-1)}{Q^{(j_0)}(z)}Q^{(j_0,l_0,m)}(z), 
\end{align}
where $\theta_{m}^{(j_0,l_0)}=(\theta^{(j_0)}_{l_0}-m)(m+\beta)$, and 
\begin{align}
\label{C_j0_l0}
C^{(j_0,l_0)}=
\begin{cases}
\dfrac{(\beta)_{l_0}}{(1+\alpha)_{l_0}}, & j_0=1, 3, \\
\dfrac{(-\alpha)_{l_0}}{(-\beta+1)_{l_0}}, & j_0=2, 4. \\
%\dfrac{(\beta)_{l_0}}{(1+\alpha)_{l_0}}, & j_0=3, \\
%\dfrac{(-\alpha)_{l_0}}{(-\beta+1)_{l_0}}, & j_0=4. 
\end{cases}
\end{align}
Since $z\mid q^{(j_0)}_{l_0}(z)$, then one observes that 
$q^{(j_0)}_{l_0}(z^{-1})Q_m(z)\in\text{span}\{z^{m-1}, \ldots, z^{-l_0-1}\}$. 
Similarly,  from (\ref{PQ_2}) and (\ref{deg_XQ0}) one can see that 
\begin{align*}
\frac{z^{-l_0-1}(z-1)}{Q^{(j_0)}(z)}Q^{(j_0,l_0,m)}(z)
\in
\begin{cases}
\text{span}\{z^{m-1}, \ldots, z^{-l_0-1}\}, & j_0=1,2, \\
\text{span}\{z^{m-1}, \ldots, z^{-l_0-2}\}, & j_0=3,4, 
\end{cases}
\end{align*}
which completes the proof of (\ref{mQ_span_12}) and (\ref{mQ_span_34}).

\end{proof}

\begin{remark}
In fact, by observing the exceptional norming constants (\ref{Xh1}) 
and the eigenvalues of quasi-Laurent-polynomial eigenfunctions in lemma \ref{lem_quasi}, we find that 
\begin{align}
\label{theta_m}
\theta_{m}^{(j_0,l_0)}=\frac{h_m^{(j_0,l_0)}}{h_m}, \quad j_0\in\{1,2,3,4\}.
\end{align}
\end{remark}

Obviously, $z^{l_0+1}\tilde{Q}^{(j_0,l_0)}_m(z)$ is a polynomial of degree $m+l_0$ for $j_0=1,2$, 
and $z^{l_0+2}\tilde{Q}^{(j_0,l_0)}_m(z)$ is a polynomial of degree $m+l_0+1$ for $j_0=3,4$. 
Then $c_{n,m}$ can be rewritten into 
\begin{align}
\label{c_nm_new}
c_{n,m}
%&=\frac{1}{h_m^{\alpha+1,\beta-1}}
%\langle P_n(z), 
%w(z^{-1})\tilde{Q}^{(j_0,l_0)}_m(z)
%\rangle
=
\begin{cases}
\frac{1}{h_m^{\alpha+1,\beta-1}}
\langle w(z)P_n(z)z^{l_0+1}, 
\tilde{Q}^{(j_0,l_0)}_m(z)z^{l_0+1}
\rangle, & j_0=1,2, \\
\frac{1}{h_m^{\alpha+1,\beta-1}}
\langle w(z)P_n(z)z^{l_0+2}, 
\tilde{Q}^{(j_0,l_0)}_m(z)z^{l_0+2}
\rangle, & j_0=3,4.
\end{cases}
\end{align}
For $j_0=1$ or $2$, it follows from lemma \ref{lem_zp_expan} that there exists constants 
$a_{n,l_0}^{(0)}(=1), a_{n,l_0}^{(1)}, \ldots, a_{n,l_0}^{(l_0+1)}$, such that 
\[
z^{l_0+1}\sum^{l_0+1}_{l=0}a_{n,l_0}^{(l)}P_{n-l}(z)
\in\textup{span}\{P_{n+l_0+1}(z), \cdots, P_{n}(z)\},
\]
where $a_{n,l_0}^{(l)}=C_{n,l_0+1}^{(l)}$, $l=1,\ldots,l_0+1$. 
Then for $0\leq m\leq n-l_0-1$, it holds that 
\[
\langle w(z)z^{l_0+1}\sum^{l_0+1}_{l=0}a_{n,l_0}^{(l)}P_{n-l}(z), 
z^{l_0+1}\tilde{Q}^{(j_0,l_0)}_m(z) \rangle =0, 
\]
which finished the proof of (\ref{expre_Coe_a_12}). 
%is equivalent to 
%\begin{align}
%\label{c_nm_linear_12}
%c_{n,m}=-a_{n,l_0}^{(1)}c_{n-1,m}-\cdots-a_{n,l_0}^{(l_0+1)}c_{n-l_0-1,m}, \quad 0\leq m\leq n-l_0-1. 
%\end{align}
However, in the case $j_0=3$ or $4$, 
there do not exist constants $a_{n,l_0}^{(0)}, \ldots, a_{n,l_0}^{(l_0+1)}$ such that 
\[
z^{l_0+2}\sum^{l_0+1}_{l=0}a_{n,l_0}^{(l)}P_{n-l}(z)\in
\text{span}\{P_{n+l_0+2}(z), \cdots, P_{n+1}(z)\}. 
\]

\section{Recurrence relations of X-HR polynomials: concrete examples}
In this section, 
examples of the recurrence relation 
\[
q^{(j_0)}_{l_0}(z)\sum^{l_0+1}_{l=0}a_l^{(j_0,l_0,n)}P^{(j_0,l_0,n-l)}(z)
=\sum^{n+l_0+1}_{j=n-l_0}b_j^{(j_0,l_0,n)}P^{(j_0,l_0,j)}(z), \quad n\geq 2l_0+1,
\]
will be provided, where the polynomials $q^{(j_0)}_{l_0}(z)$ are defined by (\ref{q_jl}). 

%\noindent
%{\textbf{Concrete example. $j_0=1, l_0=1, n=4$}}:
%\begin{align*}
%&
%q^{(1)}_{1}(z)\left(P^{(1,1,4)}(z)-\dfrac{8(4+\alpha+\beta)}{(4+\alpha)(6+\alpha)}P^{(1,1,3)}(z)
%+\dfrac{12(3+\alpha+\beta)(4+\alpha+\beta)}{(3+\alpha)(4+\alpha)(5+\alpha)(6+\alpha)}P^{(1,1,2)}(z)
%\right) \\
%&
%=\dfrac{3}{10}P^{(1,1,6)}(z)-\dfrac{2(1+\alpha-2\beta)}{(1+\alpha)(6+\alpha)}P^{(1,1,5)}(z)
%-\dfrac{-4+55\beta+17\beta^2+\alpha(-4+7\beta+\beta^2)}{2(1+\alpha)(5+\alpha)(6+\alpha)}P^{(1,1,4)}(z) \\
%&\quad
%+\dfrac{4\beta(2+\beta)(4+\alpha+\beta)}{(1+\alpha)(4+\alpha)(5+\alpha)(6+\alpha)}P^{(1,1,3)}(z), 
%\end{align*}
%where
%\[
%q^{(1)}_{1}(z)=\dfrac{z^2}{2}+\dfrac{\beta z}{1+\alpha}. 
%\]

%%%%%
As we have already mentioned, for $n\geq 2l_0+1$ and $a^{(j_0,l_0,n)}_0=1$, 
the coefficients $a^{(j_0,l_0,n)}_l$, $l=1,\ldots,l_0+1$ can be explicitly given by 
(\ref{expre_Coe_a_12}) and (\ref{expre_Coe_a_34}).

For example, let $l_0=1$ and $a^{(j_0,1,n)}_0=1$, 
then the coefficients $a^{(j_0,1,n)}_l$, $l=1,2$ can be obtained by using the following data: 
\begin{align*}
&
C^{(1)}_{n,2}=b_n+b_n^{\alpha+1,\beta-1}=-\dfrac{2n(n+\alpha+\beta)}{(n+\alpha)(n+2+\alpha)}, \\
&
C^{(2)}_{n,2}=b_n^{\alpha+1,\beta-1}b_{n-1}
=\dfrac{n(n-1)(n-1+\alpha+\beta)(n+\alpha+\beta)}{(n-1+\alpha)(n+\alpha)(n+1+\alpha)(n+2+\alpha)}, 
\end{align*}
\begin{align*}
\Xi_n^{(j_0,1)}=-(n-\theta_{1}^{(j_0)})(n+\alpha+1)
=
\begin{cases}
n-1, & j_0=1, \\
n-1+\alpha+\beta, & j_0=2, \\
n+2+\alpha+\beta, & j_0=3, \\
n+2, & j_0=4.
\end{cases}
\end{align*}

For $j_0=1$, the recurrence relation can be given explicitly as 
\begin{align*}
&
q^{(1)}_{1}(z)\bigg(P^{(1,1,n)}(z)-\dfrac{2n(n+\alpha+\beta)}{(n+\alpha)(n+2+\alpha)}P^{(1,1,n-1)}(z) \\
&\hspace{1.6cm}
+\dfrac{n(n-1)(n-1+\alpha+\beta)(n+\alpha+\beta)}{(n-1+\alpha)(n+\alpha)(n+1+\alpha)(n+2+\alpha)}P^{(1,1,n-2)}(z)
\bigg) \\
&
=\dfrac{n-1}{2(n+1)}P^{(1,1,n+2)}(z)-\dfrac{(n-2)(1+\alpha)-n\beta}{(1+\alpha)(n+2+\alpha)}P^{(1,1,n+1)}(z) \\
&\quad
-\dfrac{(n-3)n(1+\alpha)+4n^2\beta-\beta(4n+1+\alpha)(2n-1+\beta)}{2(1+\alpha)(n+1+\alpha)(n+2+\alpha)}
P^{(1,1,n)}(z) \\
&\quad
+\dfrac{n\beta(n-2+\beta)(n+\alpha+\beta)}{(1+\alpha)(n+\alpha)(n+1+\alpha)(n+2+\alpha)}P^{(1,1,n-1)}(z), 
\end{align*}
where
\[
q^{(1)}_{1}(z)=\dfrac{z^2}{2}+\frac{\beta z}{1+\alpha}. 
\]

%For example, let $l_0=1, n=5$ and $a^{(j_0,1,5)}_0=1$, 
%then the coefficients $a^{(j_0,1,5)}_l$, $l=1,2$ can be obtained by using the following data: 
%\[
%C^{(1)}_{5,2}=b_5+b_5^{\alpha+1,\beta-1}=-\dfrac{10(5+\alpha+\beta)}{(5+\alpha)(7+\alpha)}, \quad 
%C^{(2)}_{5,2}=b_5^{\alpha+1,\beta-1}b_4
%=\dfrac{20(4+\alpha+\beta)(5+\alpha+\beta)}{(4+\alpha)(5+\alpha)(6+\alpha)(7+\alpha)}, 
%\]
%\begin{align*}
%\Xi_n^{(j_0,l_0)}=-(n-\theta_{l_0}^{(j_0)})(n+\alpha+1). 
%%=
%%\begin{cases}
%%x
%%\end{cases}
%\end{align*}

In what follows, we provide more concrete examples for reference. 

\noindent
{\textbf{Example 1. $j_0=1, l_0=1, n=5$}}:
\begin{align*}
&
q^{(1)}_{1}(z)\left(P^{(1,1,5)}(z)-\dfrac{10(5+\alpha+\beta)}{(5+\alpha)(7+\alpha)}P^{(1,1,4)}(z)
+\dfrac{20(4+\alpha+\beta)(5+\alpha+\beta)}{(4+\alpha)(5+\alpha)(6+\alpha)(7+\alpha)}P^{(1,1,3)}(z)
\right) \\
&
=\dfrac{1}{3}P^{(1,1,7)}(z)-\dfrac{3+3\alpha-5\beta}{(1+\alpha)(7+\alpha)}P^{(1,1,6)}(z) \\
&\quad
-\dfrac{-10+89\beta+21\beta^2+\alpha(-10+9\beta+\beta^2)}{2(1+\alpha)(6+\alpha)(7+\alpha)}P^{(1,1,5)}(z) \\
&\quad
+\dfrac{5\beta(3+\beta)(5+\alpha+\beta)}{(1+\alpha)(5+\alpha)(6+\alpha)(7+\alpha)}P^{(1,1,4)}(z), 
\end{align*}
where
\[
q^{(1)}_{1}(z)=\dfrac{z^2}{2}+\frac{\beta z}{1+\alpha}. 
\]

\noindent
{\textbf{Example 2. $j_0=2, l_0=1, n=5$}}:
\begin{align*}
&
q^{(2)}_{1}(z)\left(P^{(2,1,5)}(z)-\dfrac{10(5+\alpha+\beta)}{(5+\alpha)(7+\alpha)}P^{(2,1,4)}(z)
+\dfrac{20(4+\alpha+\beta)(5+\alpha+\beta)}{(4+\alpha)(5+\alpha)(6+\alpha)(7+\alpha)}P^{(2,1,3)}(z)
\right) \\
&
=\dfrac{4+\alpha+\beta}{2(6+\alpha+\beta)}P^{(1,1,7)}(z)
+\dfrac{3+6\alpha+\alpha^2-2\beta-\beta^2}{(7+\alpha)(\beta-1)}P^{(2,1,6)}(z) \\
&\quad
+\dfrac{-10+3\beta+6\beta^2+\beta^3-24\alpha(4+\beta)-2\alpha^2(9+\beta)}
{2(6+\alpha)(7+\alpha)(\beta-1)}P^{(2,1,5)}(z) \\
&\quad
+\dfrac{5\alpha(3+\beta)(5+\alpha+\beta)}{(5+\alpha)(6+\alpha)(7+\alpha)(\beta-1)}P^{(2,1,4)}(z), 
\end{align*}
where
\[
q^{(2)}_{1}(z)=\dfrac{z^2}{2}+\frac{\alpha z}{-1+\beta}. 
\]

%%%
\noindent
{\textbf{Example 3. $j_0=3, l_0=1, n=5$}}:
\begin{align*}
&
q^{(3)}_{1}(z)\left(P^{(3,1,5)}(z)
-\dfrac{10(5+\alpha+\beta)(7+\alpha+\beta)}{(5+\alpha)(8+\alpha)(6+\alpha+\beta)}P^{(3,1,4)}(z)
+\dfrac{20(4+\alpha+\beta)(7+\alpha+\beta)}{(4+\alpha)(5+\alpha)(7+\alpha)(8+\alpha)}P^{(3,1,3)}(z)
\right) \\
&
=\dfrac{(6+\alpha)\beta}{2(1+\alpha)(8+\alpha)}P^{(3,1,7)}(z)
+\dfrac{(6+\alpha)(7+\alpha+\beta)(7+8\alpha+\alpha^2-4\beta-\beta^2)}
{(1+\alpha)(7+\alpha)(8+\alpha)(6+\alpha+\beta)}P^{(3,1,6)}(z) \\
&\quad
+\dfrac{140+8\beta-9\beta^2-\beta^3+4\alpha(40+7\beta)+2\alpha^2(10+\beta)}
{2(1+\alpha)(7+\alpha)(8+\alpha)}P^{(3,1,5)}(z) \\
&\quad
+\dfrac{5(4+\beta)(5+\alpha+\beta)(7+\alpha+\beta)}{(5+\alpha)(7+\alpha)(8+\alpha)(6+\alpha+\beta)}P^{(3,1,4)}(z), 
\end{align*}
where
\[
q^{(3)}_{1}(z)=\dfrac{\beta z^2}{2(1+\alpha)}+z. 
\]

\noindent
{\textbf{Example 4. $j_0=4, l_0=1, n=5$}}:
\begin{align*}
&
q^{(4)}_{1}(z)\left(P^{(4,1,5)}(z)-\dfrac{35(5+\alpha+\beta)}{3(5+\alpha)(8+\alpha)}P^{(4,1,4)}(z)
+\dfrac{28(4+\alpha+\beta)(5+\alpha+\beta)}{(4+\alpha)(5+\alpha)(7+\alpha)(8+\alpha)}P^{(4,1,3)}(z)
\right) \\
&
=\dfrac{\alpha(6+\alpha)}{2(8+\alpha)(\beta-1)}P^{(4,1,7)}(z)
-\dfrac{7(6+\alpha)(7+5\alpha-7\beta)}{6(7+\alpha)(8+\alpha)(\beta-1)}P^{(4,1,6)}(z) \\
&\quad
-\dfrac{-140+112\beta+28\beta^2+\alpha(-40+11\beta+\beta^2)}{2(7+\alpha)(8+\alpha)(\beta-1)}P^{(4,1,5)}(z) \\
&\quad
+\dfrac{35(4+\beta)(5+\alpha+\beta)}{6(5+\alpha)(7+\alpha)(8+\alpha)}P^{(4,1,4)}(z), 
\end{align*}
where
\[
q^{(4)}_{1}(z)=\dfrac{\alpha z^2}{2(-1+\beta)}+z. 
\]

\section{Concluding remarks}

In conclusion, the topic of exceptional-type extensions of COP has received 
significant attention from researchers in recent years. 
XOP generalizes COP by relaxing constraints on their degree sequence, 
and they have demonstrated their potential in various applications, 
particularly in deriving new exactly solvable potentials. 
Darboux transformations have played a crucial role in the construction of XOP, 
while other methods, such as the concept of dual families of polynomials, have also been employed. 
Properties such as recurrence relations, zeros, and spectral analysis have been extensively studied 
in the context of XOP. 
Examples of the exceptional extensions of Laurent biorthogonal polynomials, 
such as exceptional HR polynomials, have also been introduced recently by the authors, 
and their recurrence relations have been investigated in this paper. 
Unlike XOP, XLBP satisfies longer recurrence relations. 
By using the expansions of HR polynomials with twisted parameters and 
properties of the so-called forward operator and backward operators, 
we obtained the recurrence relations satisfied by the exceptional HR polynomials.

\section*{Acknowledgements}
This work was partially supported by JSPS KAKENHI Grant Numbers JP19H0179.

%\newpage
%There is a connection between these two types of recurrence relations: 
%\begin{align*}
%q^{(j_0)}_{l_0}(z)\sum^{l_0+1}_{l=0}a_l^{(j_0,l_0,n)}P^{(j_0,l_0,n-l)}(z) 
%&=\sum^{n+l_0+1}_{j=n-l_0}b_j^{(j_0,l_0,n)}P^{(j_0,l_0,j)}(z), \\
%\tilde{q}^{(j_0)}_{l_0}(z)\sum^{l_0+1}_{l=0}\tilde{a}_l^{(j_0,l_0,n)}z^{l}P^{(j_0,l_0,n-l)}(z) 
%&=\sum^{n+l_0+1}_{j=n-l_0}\tilde{b}_j^{(j_0,l_0,n)}z^{n+l_0+2-j}P^{(j_0,l_0,j)}(z). 
%\end{align*}
%In fact, 

%\newpage


\begin{thebibliography}{}

\bibitem{XOPviaKrall}
A. J. Dur\'an, 
``Exceptional orthogonal polynomials via Krall discrete polynomials", 
Lectures on Orthogonal Polynomials and Special Functions, 464: 1, (2020).

\bibitem{BispOP}
A. J. Dur\'an and M. D. de la Iglesia, 
``Constructing bispectral orthogonal polynomials from the classical discrete families of Charlier, 
Meixner and Krawtchouk", 
%Constructive Approximation, 
Constr. Approx., 2015, 41(1): 49-91.

\bibitem{Krall_Hahn}
A. J. Dur\'an and D. Manuel, 
``Constructing Krall–Hahn orthogonal polynomials", 
%Journal of Mathematical Analysis and Applications, 
J. Math. Anal. Appl., 2015, 424(1): 361-384.

\bibitem{X-Bochner}
M. Á. García-Ferrero, D. G\'omez-Ullate and R. Milson, 
``A Bochner type characterization theorem for exceptional orthogonal polynomials", 
%Journal of Mathematical Analysis and Applications, 
J. Math. Anal. Appl., 2019, 472(1): 584-626.

\bibitem{GGM13}
D. G\'omez-Ullate, Y. Grandati and R. Milson, 
``Rational extensions of the quantum harmonic oscillator and exceptional Hermite polynomials", 
%Journal of Physics A: Mathematical and Theoretical, 
J. Phys. A Math. Theor., 2013, 47(1): 015203.

\bibitem{GKM09}
D. G\'omez-Ullate, N. Kamran, and R. Milson, 
``An extended class of orthogonal polynomials defined by a Sturm–Liouville problem", 
%Journal of Mathematical Analysis and Applications, 
J. Math. Anal. Appl., 2009, 359, 352–367.

\bibitem{GKM10_1}
D. G\'omez-Ullate, N. Kamran and R. Milson,  
``An extension of Bochner's problem: exceptional invariant subspaces",
%Journal of Approximation Theory, 
J. Approx. Theor., 2010, 162(5): 987-1006.

\bibitem{GKM10_2}
D. G\'omez-Ullate, N. Kamran, and R. Milson, 
``Exceptional orthogonal polynomials and the Darboux transformation", 
%Journal of Physics A: Mathematical and Theoretical, 
J. Phys. A Math. Theor., 2010, 43(43): 434016.

\bibitem{GKK16}
D. G\'omez-Ullate, A. Kasman, A. B. J. Kuijlaars and R. Milson, 
``Recurrence relations for exceptional Hermite polynomials", 
%Journal of Approximation Theory, 
J. Approx. Theor., 2016, 204: 1-16.

\bibitem{GKM12}
D. G\'omez-Ullate, N. Kamran, and R. Milson, 
``Two-step Darboux transformations and exceptional Laguerre polynomials", 
%Journal of Mathematical Analysis and Applications, 
J. Math. Anal. Appl., 2012, 387(1): 410-418.

\bibitem{GMM13}
D. G\'omez-Ullate, F. Marcell\'an and R. Milson, 
``Asymptotic and interlacing properties of zeros of exceptional Jacobi and Laguerre polynomials", 
%Journal of Mathematical Analysis and Applications, 
J. Math. Anal. Appl., 2013, 399(2): 480-495.

\bibitem{HR}
E. Hendriksen and H. van Rossum, 
``Orthogonal Laurent polynomials", 
Indagationes Mathematicae (Proceedings). North-Holland, 1986, 89(1): 17-36.

\bibitem{HS12}
C. L. Ho and R. Sasaki, 
``Zeros of the exceptional Laguerre and Jacobi polynomials", 
International Scholarly Research Notices, 2012, 2012.

\bibitem{KM15}
A. B. J. Kuijlaars and R. Milson, 
``Zeros of exceptional Hermite polynomials", 
%Journal of Approximation Theory, 
J. Approx. Theor., 2015, 200: 28-39.

\bibitem{LLM16}
C. Liaw, L. L. Littlejohn, R. Milson R, et al. 
``The spectral analysis of three families of exceptional Laguerre polynomials", 
%Journal of Approximation Theory, 
J. Approx. Theor., 2016, 202: 5-41.

\bibitem{XBI}
Y. Luo and S. Tsujimoto, 
``Exceptional Bannai–Ito polynomials", 
%Journal of Approximation Theory, 
J. Approx. Theor., 2019, 239: 144-173.

\bibitem{XHR}
Y. Luo and S. Tsujimoto, 
``Exceptional Laurent biorthogonal polynomials through spectral transformations of generalized eigenvalue problems", 
Stud. Appl. Math., 2023, 151(2): 585-615.

\bibitem{MR09}
B. Midya and B. Roy, 
``Exceptional orthogonal polynomials and exactly solvable potentials in position dependent mass Schrödinger Hamiltonians", 
%Physics Letters A, 
Phys. Lett. A, 2009, 373(45): 4117-4122.

\bibitem{X-RR}
H. Miki and S. Tsujimoto, 
``A new recurrence formula for generic exceptional orthogonal polynomials", 
%Journal of Mathematical Physics, 
J. Math. Phys., 2015, 56(3): 033502.

\bibitem{X-Kra}
H. Miki, S. Tsujimoto and L. Vinet, 
``The single-indexed exceptional Krawtchouk polynomials”, 
%Journal of Difference Equations and Applications, 
J. Differ. Equ. Appl., 2023, 29(3): 344-365.

\bibitem{O16}
S. Odake, 
``Recurrence relations of the multi-indexed orthogonal polynomials. III", 
%Journal of Mathematical Physics, 
J. Math. Phys., 2016, 57(2): 023514.

\bibitem{OS13}
S. Odake and R. Sasaki, 
``Multi-indexed Wilson and Askey-Wilson polynomials", 
%Journal of Physics A: Mathematical and Theoretical, 
J. Phys. A Math. Theor., 2013, 46(4): 045204.

\bibitem{OS09}
S. Odake and R. Sasaki, 
``Infinitely many shape invariant discrete quantum mechanical systems and new exceptional orthogonal polynomials related to the Wilson and Askey-Wilson polynomials", 
%Physics Letters B, 
Phys. Lett. B, 2009, 682(1): 130-136.

\bibitem{OS11}
S. Odake and R. Sasaki, 
``Exactly solvable quantum mechanics and infinite families of multi-indexed orthogonal polynomials", 
%Physics Letters B, 
Phys. Lett. B, 2011, 702(2-3): 164-170. 

\bibitem{X q-Racah}
S. Odake and R. Sasaki, 
``The exceptional ($X_l$)($q$)-Racah polynomials", 
%Progress of theoretical physics, 
Progr. Theoret. Phys., 2011, 125(5): 851-870.

\bibitem{Quesne08}
C. Quesne, 
``Exceptional orthogonal polynomials, exactly solvable potentials and supersymmetry", 
%Journal of Physics A: Mathematical and Theoretical, 
J. Phys. A Math. Theor., 2008, 41(39): 392001.

\bibitem{Quesne09}
C. Quesne, 
``Solvable rational potentials and exceptional orthogonal polynomials in supersymmetric quantum mechanics", 
%SIGMA. Symmetry, Integrability and Geometry: Methods and Applications, 
SIGMA Symmetry Integrability Geom. Methods Appl., 2009, 5: 084.

\bibitem{STA}
R. Sasaki, S. Tsujimoto and A. Zhedanov, 
``Exceptional Laguerre and Jacobi polynomials and the corresponding potentials through Darboux-Crum transformations", 
%Journal of Physics A: Mathematical and Theoretical 
J. Phys. A Math. Theor., 2010, 43, 315204.





%\bibitem{BORF_GEVP}
%A. Zhedanov, 
%``Biorthogonal rational functions and the generalized eigenvalue problem", 
%%Journal of Approximation Theory, 
%J. Approx. Theor., 1999, 101(2): 303-329.

%\bibitem{A82}
%R. Askey, 
%``Discussion of Szeg\"o’s paper ``Beiträge zur Theorie der Toeplitzschen Formen" ", 
%in Gabor Szegö: Collected works, Vol. 1 (Birkha\"user, Boston-Basel-Stuttgart, 1982), 303–305.
%
%\bibitem{A85}
%R. Askey, 
%``Some problems about special functions and computations", 
%Rend. Semin. Mat. Univ. Politec. Torino 1– 22 (1985), 1 –22.
%
%\bibitem{CLBP}
%A. Zhedanov, 
%``The “classical” Laurent biorthogonal polynomials", 
%%Journal of computational and applied mathematics, 
%J. Comput. Appl. Math., 1998, 98(1): 121-147.
%
%\bibitem{GVZ04}
%F. A. Gr\"unbaum, L. Vinet and A. Zhedanov, 
%``Linear operator pencils on Lie algebras and Laurent biorthogonal polynomials", 
%%Journal of Physics A: Mathematical and General, 
%J. Phys. A Math. Gen., 2004, 37(31): 7711.

\bibitem{ABP}
L. Vinet and A. Zhedanov, 
``An algebraic treatment of the Askey biorthogonal polynomials on the unit circle", 
Forum of Mathematics, Sigma. Cambridge University Press, 2021, 9.

%\bibitem{HR_1}
%E. Hendriksen and O. Njåstad, 
%``Biorthogonal Laurent polynomials with biorthogonal derivatives", 
%%The Rocky Mountain Journal of Mathematics, 
%Rocky. MT. J. Math.,1991: 301-317.
%
%\bibitem{PASI}
%A. Zhedanov,  
%``On the polynomials orthogonal on regular polygons", 
%%Journal of Approximation theory, 
%J. Approx. Theor., 1999, 97(1): 1-14.







%\bibitem{GM18}
%D. G\'omez-Ullate and R. Milson, 
%``Exceptional orthogonal polynomials and rational solutions to Painlevé equations", 
%AIMS-Volkswagen Stiftung Workshops. Birkhäuser, Cham, 2018: 335-386.


%\bibitem{Koekoek}
%R. Koekoek, P. A. Lesky, and R. F. Swarttouw, 
%``Hypergeometric Orthogonal Polynomials and Their $q$-Analogues", 
%Springer Monographs in Mathematics (Springer, 2010). 

%\bibitem{Pastro}
%P. I. Pastro, 
%``Orthogonal polynomials and some q-beta integrals of Ramanujan", 
%Journal of mathematical analysis and applications, 1985, 112(2): 517-540.


%\bibitem{EHF}
%H. Rosengren, 
%``Elliptic hypergeometric functions", 
%Lectures on Orthogonal Polynomials and Special Functions, 2016.

%\bibitem{newBORF}
%V. Spiridonov and A. Zhedanov, 
%``Spectral transformation chains and some new biorthogonal rational functions", 
%Communications in Mathematical Physics, 2000, 210(1): 49-83.
%
%\bibitem{theta_Hyper}
%V. Spiridonov, 
%``Theta hypergeometric integrals", 
%St. Petersburg Mathematical Journal, 2004, 15(6): 929-967.
%
%\bibitem{Essays}
%V. P. Spiridonov, 
%``Essays on the theory of elliptic hypergeometric functions", 
%Russian Mathematical Surveys, 2008, 63(3): 405.
%
%\bibitem{LBP_Pastro}
%L. Vinet and A. Zhedanov, 
%``Spectral transformations of the Laurent biorthogonal polynomials, II. Pastro polynomials", 
%Canadian Mathematical Bulletin, 2001, 44(3): 337-345.


%\bibitem{GEVP_DT}
%A. Zhedanov, 
%``Regular algebras of dimension 2, the generalized eigenvalue problem and Pad\'e interpolation", 
%Journal of Nonlinear Mathematical Physics, 12(sup2): 333-356 (2005). 


\end{thebibliography}
\end{document}